\documentclass[12pt]{amsart}
\usepackage{pifont}
\usepackage{txfonts}
\usepackage{}
\usepackage{amsfonts}
\usepackage{graphicx}
\usepackage{epstopdf}
\usepackage{amsmath}
\usepackage{amsthm}
\usepackage{latexsym,bm}
\usepackage{amssymb}
\usepackage{esint}
\usepackage{enumitem}
\usepackage{indentfirst}
\usepackage{amscd}
\newtheorem{thm}{Theorem}[section]
\newtheorem{cor}[thm]{Corollary}
\newtheorem{lem}[thm]{Lemma}

\newtheorem{defn}[thm]{Definition}

\newtheorem{Remark}{Remark}
\numberwithin{equation}{section}
\numberwithin{Remark}{section}

\begin{document}

\title{On Uniqueness And Existence of Conformally Compact Einstein Metrics with Homogeneous Conformal Infinity. II}

\author{Gang Li$^\dag$}

\begin{abstract} In this paper we show that for an $\text{Sp}(k+1)$ invariant metric $\hat{g}$ on $\mathbb{S}^{4k+3}$ $(k\geq 1)$ close to the round metric, the conformally compact Einstein (CCE) manifold $(M, g)$ with $(\mathbb{S}^{4k+3}, [\hat{g}])$ as its conformal infinity is unique up to isometries. Moreover, by the result in \cite{LiQingShi}, $g$ is the Graham-Lee metric (see \cite{GL}) on the unit ball $B_1\subset \mathbb{R}^{4k+4}$. We also give an a priori estimate on the Einstein metric $g$. As a byproduct of the a priori estimates, based on the estimate and Graham-Lee and Lee's  seminal perturbation result (see \cite{GL} and \cite{Lee}), we use the continuity method directly to obtain an existence result of the non-positively curved CCE metric with prescribed conformal infinity $(\mathbb{S}^{4k+3}, [\hat{g}])$ when the metric $\hat{g}$ is $\text{Sp}(k+1)$-invariant. We also generalize the results to the case of conformal infinity $(\mathbb{S}^{15},[\hat{g}])$ with $\hat{g}$ a Spin$(9)$-invariant metric in the appendix.
\end{abstract}

\renewcommand{\subjclassname}{\textup{2000} Mathematics Subject Classification}
 \subjclass[2010]{Primary 53C25; Secondary 58J05, 53C30, 34B15}


\thanks{$^\dag$ Research partially supported by the National Natural Science Foundation of China No. 11701326, the Fundamental Research Funds of Shandong University 2016HW008 and the Young Scholars Program of Shandong University 2018WLJH85.}

\address{Gang Li, School of Mathematics, Shandong University, Jinan, Shandong Province, China}
\email{runxing3@gmail.com}

\maketitle


\section{Introduction}



This is a continuation of our previous work \cite{Li}\cite{Li2} on uniqueness and existence of conformally compact Einstein (CCE) metrics (see Definition \ref{defn_conformallycompactEinsteinmetric}) with prescribed homogeneous conformal infinity. Let $B_1$ be the unit ball in the Euclidean space $\mathbb{R}^{n+1}$ of dimension $(n+1)$, with the boundary $\mathbb{S}^n$. Given an $\text{Sp}(k+1)$ invariant metric $\hat{g}$ on $\mathbb{S}^n$ with $n=4k+3$ $(k\geq1)$, in this paper we consider the uniqueness and existence of non-positively curved CCE metrics $g$ on $B_1$ with $(\mathbb{S}^n, [\hat{g}])$ as its conformal infinity.

In \cite{GL}, for any Riemannian metric $\hat{g}$ on the $n$-sphere $\mathbb{S}^n$ which is $C^{2,\alpha}$ close to the round metric, with the aid of a gauge fixing technique and a construction of approximation solutions, Graham and Lee used the Fredholm theory of elliptic operators on certain weighted spaces to prove the existence of a CCE metric on the $(n+1)$-ball $B_1(0)$ with $(\mathbb{S}^n, [\hat{g}])$ as its conformal infinity, with the solution unique in a small neighborhood of the asymptotic solution they constructed in the weighted space. Later Lee \cite{Lee} generalized this perturbation result to more general CCE manifolds which include the case when they are non-positively curved. When $\hat{g}$ is a homogeneous metric near the round sphere metric, Biquard \cite{Biquard1} used harmonic analysis on $(\mathbb{S}^n,\hat{g})$ to give an elementary proof of the perturbation result. Given the conformal infinity $(\mathbb{S}^3,[\hat{g}])$ with $\hat{g}$ an $\text{SU}(2)$ invariant metric, Pedersen \cite{Pedersen} and Hitchin \cite{Hitchin1} filled in a global CCE metric on the $4$-ball, which has self-dual Weyl curvature, and the metric is unique under the self-duality assumption. Anderson \cite{Anderson} and Biquard \cite{Biquard2} proved that the CCE metric is unique up to isometry, provided that both the conformal infinity $(\mathbb{S}^n, [\hat{g}])$ and the non-local term in the expansion (see Theorem \ref{thm_expansion1}) of the Einstein metric at infinity are given. For more of the existence result, one refers to \cite{Anderson1}\cite{LeBrun1}\cite{FG}\cite{SKichenassamy}\cite{GS}, etc. When $\hat{g}$ is the round metric, it is proved that the CCE metric filled in must be the hyperbolic space, see \cite{Andersson-Dahl}\cite{Q}\cite{DJ}\cite{LiQingShi}, see also \cite{CLW}. Non-existence of CCE with various given data as conformal infinity is discussed in \cite{HP}\cite{Anderson}\cite{Wang}\cite{GH}\cite{GHS}, etc. For more references, one refers to \cite{GHS} and \cite{Li2}.

Recall that in \cite{Li} and \cite{Li2}, for a homogeneous metric $\hat{g}$ on $\mathbb{S}^n$, the problem of solving a non-positively curved conformally compact Einstein metric with the prescribed conformal infinity $(\mathbb{S}^n, [\hat{g}])$ is deformed to a two-point boundary value problem of a system of ordinary differential equations on the interval $x\in[0,1]$. The uniqueness of the non-positively curved solution was proved (see \cite{Li} \cite{Li2}) for $\hat{g}$ on $(\mathbb{S}^n,[\hat{g}])$ when $\hat{g}$ is a Berger metric on $\mathbb{S}^3$, an $\text{SU}(k+1)$-invariant metric on $\mathbb{S}^{2k+1}$ $(k\geq2)$ and a generalized Berger metric not far from the round metric on $\mathbb{S}^3$; while the global uniqueness of the CCE metrics was proved for all these three classes of homogeneous metrics on $\mathbb{S}^n$ which are close to the round sphere metric. In \cite{Li2}, for an $\text{Sp}(k+1)$ invariant metric $\hat{g}$ on $\mathbb{S}^n$ with $n=4k+3$, the boundary value problem is deformed to the boundary value problem $(\ref{equn_SpnEinstein01})-(\ref{equn_SpnBV01})$ of the functions $y_i=y_i(x)$ on $x\in[0,1]$ for $1\leq i\leq 4$, provided that the CCE filling-in is non-positively curved. In this paper, we aim to study the uniqueness of the solutions to this boundary value problem.  The main difficulty here in comparison to the three cases discussed before is that $y_i=y_i(x)$ fails to be monotone on $x\in[0,1]$ in general for $2\leq i\leq 4$ (see Lemma \ref{lem_Spn2monotonicityt1t22}). Recall that for the three cases dealt in \cite{Li} and \cite{Li2}, the monotonicity of $y_i$ is the starting point for the estimate of $y_i$ on $x\in[0,1]$, which gives the $C^0$ estimate of the solution and then by the Einstein equations we could derive the $C^k$ estimate of the solution. To handle this difficulty, we find that $y_i-y_j$ is monotone for $2\leq i<j\leq 4$. This combining with the Einstein equations gives the uniform bound of $y_i$. Based on this, we give an a priori estimate of the solution, and then as in \cite{Li2} we obtain Theorem \ref{thm_someSpnmetric}, an existence result of the non-positively curved conformally compact Einstein metrics using a direct continuity method.

 Using the a prior estimate on the solution and smallness of $\sup_M|W|_g$, we obtain the main result on the uniqueness of the solution.

\begin{thm}\label{thm_uniquenessSpn}
Let $\hat{g}$ be a homogeneous metric on $\mathbb{S}^n\cong \text{Sp}(k+1)/\text{Sp}(k)$ with $n=4k+3$ for $k\geq 1$ so that $\hat{g}$ has the standard diagonal form
\begin{align}\label{eqn_Spnstandardmetricform_thm}
\hat{g}=\lambda_1\sigma_1^2+\lambda_2\sigma_2^2+\lambda_3\sigma_3^2+\lambda_4(\sigma_4^2+..+\sigma_n^2),
\end{align}
at a point $q\in \mathbb{S}^n$, where $\lambda_i$ ($1\leq i \leq 4$) is a positive constant and $\sigma_1,..,\sigma_n$ are the $1$-forms with respect to the basis vectors in $\mathfrak{p}$, in the $\text{Ad}_{\text{Sp}(k)}$-invariant splitting $\text{sp}(k+1)=\text{sp}(k)\oplus \mathfrak{p}$. Assume that $\frac{\lambda_i}{\lambda_j}$ ($1\leq i, j\leq 4$) is close enough to $1$, then up to isometry the conformally compact Einstein metric filled in is unique and it is the perturbation metric in \cite{GL} on the $(n+1)$-ball $B_1(0)$ with $(\mathbb{S}^n, [\hat{g}])$ as its conformal infinity.
\end{thm}
For Theorem \ref{thm_uniquenessSpn}, we argue by contradiction. For two solutions $(y_{11},..,y_{14})$ and $(y_{21},..,y_{24})$, denote $z_i=y_{1i}-y_{2i}$. We consider the total variation of $z_i$ on $x\in[0,1]$, as in \cite{Li2}. The difference here is that since $y_i$ is not monotone, $|y_i(x)|$ is not uniformly controlled by $|y_i(0)|$. The smallness of $\sup_M|W|_g$ plays an important role in dealing with this issue. As an easy consequence of the a priori estimates in the proof of the uniqueness result, we obtain an existence result of the CCE metrics with prescribed conformal infinity $(\mathbb{S}^{4k+3},[\hat{g}])$ with $\hat{g}$ an $\text{Sp}(k+1)$-invariant metric, see Theorem \ref{thm_someSpnmetric}.

The uniqueness and the existence results are extended to the case when the conformal infinity $(\mathbb{S}^{15}, [\hat{g}])$ with $\hat{g}$ a $\text{Spin}(9)$-invariant metric in the appendix, see Theorem \ref{thm_someSpinmetric} and Theorem \ref{thm_someSpinmetric1}, and in the system of ODEs obtained there are only two unknown functions.

Remark that recently, in \cite{CGQ}\cite{CGJQ}, the authors obtained uniqueness of CCE metric for general prescribed conformal infinity data which is close to the round $n$-sphere, which is based on a compactness argument for the Fefferman-Graham conformal compactification metrics in this situation. Our approach here for the conformal infinity which is homogeneous spheres is completely different from theirs, and shows a possibility for the uniqueness of CCE metrics which the prescribed conformal infinity is not that close to the round sphere. In \cite{Chi}, the author obtained a relatively general existence result of CCE metrics with conformal infinity $(\mathbb{S}^{4k+3}, [\hat{g}])$ with $\hat{g}$ $\text{Sp}(k+1)$-invariant by solving a dynamic system starting from the interior, but due to the complexity of the behavior of the system at infinity he could not determine precisely the conformal class $[\hat{g}]$.


The organization of the paper is as follows. In Section \ref{Sect:preliminary}, we present a two-point boundary value problem of the system of ODEs $(\ref{equn_SpnEinstein01})-(\ref{equn_SpnBV01})$ deformed from the boundary value problem of the Einstein equations and list some results about CCE manifolds used in this paper. In Section \ref{section4}, we start with the monotonicity of $y_1$ and $y_i-y_j$ with $2\leq i < j \leq 4$, see Lemma \ref{lem_monotonicitySpn01}. After that, we give uniform upper bound and lower bound of the solutions, see Corollary \ref{cor_Spnboundivd}, and hence we can then follow the estimate in \cite{Li2}. Using the Einstein equations, we then obtain an a priori estimate of the solution away from $x=1$ (see Lemma \ref{lem_Spnuniformests301}), and with the aid of the uniform bound of $\sup_M|W|_g$ when $g$ is non-positively curved we obtain the $C^2$ estimate of the solution away from $x=0$ (see Lemma \ref{lem_SpnyiboundrightGB}). A direct continuity argument gives the existence result in Theorem \ref{thm_someSpnmetric}, based on the a priori estimate and Graham-Lee and Lee's perturbation result.

In Section \ref{sect_Spnsolutionestimates}, we consider the special case $(\ref{equn_Spn1Einstein01})-(\ref{equn_Spn1BV01})$ with the unknown functions $y_1$ and $y_2$ when $\hat{g}$ is $\text{Sp}(k+1)\times \text{Sp}(1)$ invariant. Lemma \ref{lem_Spn1monotonicityt1} shows that for $\sup_M|W|_g$ relatively small, $y_2$ is still monotone, and the argument on the existence and uniqueness of the solutions for the given $\text{SU}(k+1)$ invariant conformal infinity in \cite{Li2} still holds here. In Section \ref{section5}, we consider the special case $(\ref{equn_Spn2Einstein01})-(\ref{equn_Spn2BV01})$ with the unknown functions $y_1,y_2$ and $y_3$ when $\hat{g}$ is $\text{Sp}(k+1)\times \text{SU}(1)$ invariant. It is pointed in Lemma \ref{lem_Spn2monotonicityt1t22} that $y_2$ is not monotone on $x\in[0,1]$ in general, but a uniform control of the solution by the initial data still holds when $\sup_M|W|_g$ is relatively small.

In Section \ref{section6}, we prove Theorem \ref{thm_uniquenessSpn} based on the a priori estimate in Section \ref{section4}, following the approach of the generalized Berger metric case in \cite{Li2}. Let $(y_{11},..,y_{14})$ and $(y_{21},..,y_{24})$ be two solutions to $(\ref{equn_SpnEinstein01})-(\ref{equn_SpnBV01})$, and denote $z_i=y_{1i}-y_{2i}$ for $1\leq i \leq 4$. We consider the total variation of $z_i$. When the initial data is not far from that of the round metric and the Weyl tensor satisfies that $\sup_M|W|_g$ is relatively small, using certain integration of the equations satisfied by $z_i$ ($1\leq i\leq 4$), we derive a control of the total variation of $z_i$ on some "good" intervals of monotonicity of $z_i$ by the linear combination of $\{z_1,..,z_4\}-\{z_i\}$ with small coefficients. Notice that the smallness of the coefficients dues to the smallness of $\sup_M|W|_g$ at the points away from $x=0$, which is different from \cite{Li2}, since here we do not have the monotonicity of $y_i$ in general. By summarizing the inequalities on these intervals, we obtain the control of the total variation $V(z_i)$ by the linear combination of $\{V(z_1),..,V(z_4)\}-\{V(z_i)\}$ with small coefficients for each $i$, and hence we obtain that $z_i$ vanishes identically for $1\leq i \leq 4$. Notice that $z_1$ satisfies a different type of equation to that of $z_i$ $(2\leq i\leq 4)$, and the discussion on $z_1$ in this approach is different from the other three.

Finally, in Appendix \ref{Appendix1}, we first derive the two-point boundary value problem of a system of ODEs which is equivalent to the non-positively curved CCE filling-in problem with the prescribed conformal infinity $(\mathbb{S}^{15},[\hat{g}])$, where $\hat{g}$ is Spin$(9)$-invariant. Then we generalize the uniqueness and existence argument to this case.

\vskip0.2cm
{\bf Acknowledgements.} The author would like to thank Professor Jie Qing and Professor Yuguang Shi for helpful discussion and constant support. The author is also grateful to Professor Alice Chang and Professor Matthew Gursky for their interest and encouragement.  

\section{Preliminaries}\label{Sect:preliminary}

\begin{defn}\label{defn_conformallycompactEinsteinmetric}
Suppose $\overline{M}$ is a smooth compact manifold of dimension $n+1$ with boundary $\partial M$ and $M$ is the interior of $\overline{M}$. A defining function $x$ on $\overline{M}$ is a smooth function $x$ on $\overline{M}$ such that
\begin{align*}
x>0\,\,\text{ in}\,\, M,\,\, x=0\,\,\text{and}\,\,dx\neq 0\,\, \text{on}\,\,\partial M.
\end{align*}
A complete Riemannian metric $g$ on $M$ is conformally compact if there exists a defining function $x$ such that $x^2g$ extends by continuity to a Riemannian metric $\bar{g}$ (of class at least $C^0$) on $\overline{M}$. The rescaled metric $\bar{g}=x^2g$ is called a conformal compactification of $g$. If for some smooth defining function $x$, $\bar{g}$ is in $C^k(\overline{M})$ or the Holder space $C^{k,\alpha}(\overline{M})$, we say $g$ is conformally compact of class $C^k$ or $C^{k, \alpha}$. Moreover, if $g$ is also Einstein, i.e.,
 \begin{align}\label{eqn_Einstein}
Ric_g=-n g,
\end{align}
we call $g$ a conformally compact Einstein (CCE) metric. Also, for the restricted metric $\hat{g}=\bar{g}\big|_{\partial M}$, the conformal class $(\partial M, [\hat{g}])$ is called the conformal infinity of $(M, g)$. A defining function $x$ is called a geodesic defining function about $\hat{g}$ if $\hat{g}=\bar{g}\big|_{\partial M}$ and $|dx|_{\bar{g}}=1$ in a neighborhood of the boundary.
\end{defn}

Let $(M^{n+1}, g)$ be a non-positively curved simply connected CCE manifold with its conformal infinity $(\partial M, [\hat{g}])$. By the non-positivity of the sectional curvature of $g$ and $(\ref{eqn_Einstein})$, we have that $|W|_g\leq \sqrt{(n^2-1)n}$ pointwisely. When $(M, g)$ is not the hyperbolic space, it is shown in \cite{Li} that there exists a unique point $p_0\in M$ which is the center of the unique closed geodesic
ball of the smallest radius that contains the set $S\equiv\{p\in M \big| |W|_g(p)=\sup_M|W|_g\}$. $p_0$ is called the {\it (spherical) center of
gravity of } $(M, g)$. Each conformal Killing vector field $Y$ on $(\partial M, [\hat{g}])$ can be extended continuously to a Killing vector field $X$ on $(M,g)$. Under the action generated by $X$, $p_0$ is a fixed point and hence each geodesic sphere centered at $p_0$ is an invariant subset (see \cite{Li}). The Einstein metric $g$ has the orthogonal splitting
\begin{align}\label{eqn_metrictwocomponents}
g=dr^2+g_r,
\end{align}
where $r$ is the distance function to $p_0$ and $g_r$ is the restriction of $g$ on the $r$-geodesic sphere centered at $p_0$. If moreover, $(\partial M, \hat{g})$ is a homogeneous space, then the function $x=e^{-r}$ is a geodesic defining function about $C\hat{g}$ with $C>0$ some constant, i.e. $\hat{g}=C\displaystyle\lim_{x\to 0}(x^2g_r)$. The metric can then be expressed as
\begin{align}\label{equn_splittingmetric01}
g=dr^2+\sinh^2(r)\bar{h}=x^{-2}(dx^2+\frac{(1-x^2)^2}{4}\bar{h}),
\end{align}
for $0\leq x \leq 1$. Let $(r, \theta)$ be the polar coordinates centered at $p_0$.

Let $\hat{g}$ be a homogeneous metric on $\mathbb{S}^n\cong \text{Sp}(k+1)/\text{Sp}(k)$ (here $\text{Sp}(k)$ is the symplectic group of $2k$
 variables) with $n=4k+3$ for $k\geq 1$ so that $\hat{g}$ has the standard diagonal form
\begin{align}\label{equn_Spnbvp}
\hat{g}=\lambda_1\sigma_1^2+\lambda_2\sigma_2^2+\lambda_3\sigma_3^2+\lambda_4(\sigma_4^2+..+\sigma_n^2),
\end{align}
at a point $q\in \mathbb{S}^n$, where $\lambda_i$ ($1\leq i \leq 4$) is a positive constant and $\sigma_1,..,\sigma_n$ are the $1$-forms with respect to the basis vectors in an irreducible subspace $\mathfrak{p}$ in the $\text{Ad}_{\text{Sp}(k)}$-invariant splitting $\text{sp}(k+1)=\text{sp}(k)\oplus \mathfrak{p}$.  Here $\text{sp}(k)$ is the Lie algebra of $\text{Sp}(k)$.  Let $\theta=(\theta^1,...,\theta^n)$ be a local coordinate on $\mathbb{S}^n$ near the point $q$ such that $\theta=0$ and $d\theta^i=\sigma_i$ at $q$.It is shown in \cite{Li2} by symmetry extension that along the geodesic $\gamma$ connecting $q$ and $p_0$, the metrics $\bar{h}$ in $(\ref{equn_splittingmetric01})$ restricted on the geodesic spheres have the diagonal form
\begin{align*}
\bar{h}=I_1(x)d(\theta^1)^2+I_2(x)d(\theta^2)^2+I_3(x)d(\theta^3)^2+ I_4(x)(d(\theta^4)^2+...+d(\theta^n)^2),
\end{align*}
at any point $(x, 0)$ under the coordinate $(x,\theta)$ for $0 \leq x \leq 1$, with some positive functions $I_i\in C^{\infty}([0,1])$ satisfying $I_i(1)=1$, for $1\leq i\leq4$. 

Denote $K=I_1I_2I_3I_4^{n-3},\,\,t_i=\frac{I_i}{I_4}$ for $1\leq i\leq 3$. Let $y_1=\log(K)$ and $y_{i+1}=\log(t_i)$ for $1\leq i\leq3$. It was shown in \cite{Li2} that for the $\text{Sp}(k+1)$ invariant metric $\hat{g}$ of the standard form $(\ref{equn_Spnbvp})$ at $q$ on $\mathbb{S}^n$, the Einstein equations $(\ref{eqn_Einstein})$ with the prescribed conformal infinity $(\mathbb{S}^n, [\hat{g}])$ is equivalent to the boundary value problem of a system of ordinary differential equations:
\begin{align}
&\label{equn_SpnEinstein01}y_1''-x^{-1}(1+3x^2)(1-x^2)^{-1}y_1'+\frac{1}{2n^2}[n(y_1')^2\,+\,((n-1)y_2'-y_3'-y_4')^2\\
&+\,(-y_2'+(n-1)y_3'-y_4')^2\,+\,(-y_2'-y_3'+(n-1)y_4')^2+(n-3)(y_2'+y_3'+y_4')^2]=0,\notag\\
&\label{equn_SpnEinstein02}y_1''-x^{-1}(2n-1+(2n+1)x^2)(1-x^2)^{-1}y_1'+\frac{1}{2}(y_1')^2\\
&+\,8(1-x^2)^{-2}[\,n(n-1)-(K^{-1}t_1t_2t_3)^{\frac{1}{n}}\,\big(\,(n-3)(n+5)-(n-3)(t_1+t_2+t_3)\notag\\
&+\frac{2(2t_1t_2+2t_1t_3+2t_2t_3-t_1^2-t_2^2-t_3^2)}{t_1t_2t_3}\,\big)\,]=0,\notag\\
&\label{equn_SpnEinstein03}y_2''-x^{-1}(n-1+(n+1)x^2)(1-x^2)^{-1}y_2'+\frac{1}{2}y_1'y_2'\\
&- 8 (1-x^2)^{-2} (K^{-1}t_1t_2t_3)^{\frac{1}{n}}[(n-1)t_1+2t_2+2t_3-n-5+\frac{2(t_1^2-(t_2-t_3)^2)}{t_1t_2t_3}]=0,\notag\\
&\label{equn_SpnEinstein04}y_3''-x^{-1}(n-1+(n+1)x^2)(1-x^2)^{-1}y_3'+\frac{1}{2}y_1'y_3'\\
&-8 (1-x^2)^{-2} (K^{-1}t_1t_2t_3)^{\frac{1}{n}}[(n-1)t_2+2t_1+2t_3-n-5+\frac{2(t_2^2-(t_1-t_3)^2)}{t_1t_2t_3}]=0,\notag\\
&\label{equn_SpnEinstein05}y_4''-x^{-1}(n-1+(n+1)x^2)(1-x^2)^{-1}y_{4}'+\frac{1}{2}y_1'y_4'\\
&-8 (1-x^2)^{-2} (K^{-1}t_1t_2t_3)^{\frac{1}{n}}[(n-1)t_3+2t_1+\,2t_2-n-5+\frac{2(t_3^2-(t_1-t_2)^2)}{t_1t_2t_3}]=0,\notag
\end{align}
for $y_k(x)\in C^{\infty}([0,1])$ ($1\leq k \leq 4$), with the boundary condition
\begin{align}\label{equn_SpnBV01}
t_i(0)=\frac{\lambda_i}{\lambda_4},\,\,K(1)=t_i(1)=1,\,\,y_j'(0)=y_j'(1)=0,
\end{align}
for $1\leq i \leq 3$ and $1\leq j \leq 4$. Combining $(\ref{equn_SpnEinstein01})$ and $(\ref{equn_SpnEinstein02})$, we have
\begin{align}
&\label{equn_SpnEinstein06}(y_1')^2-4nx^{-1}(1+x^2)(1-x^2)^{-1}y_1'
-\frac{1}{n(n-1)}[((n-1)y_2'-y_3'-y_4')^2+(-y_2'+(n-1)y_3'-y_4')^2\\
&+(-y_2'-y_3'+(n-1)y_4')^2+(n-3)(y_2'+y_3'+y_4')^2]
+\frac{16n}{n-1}(1-x^2)^{-2}\,[\,n(n-1)\notag\\
&-(K^{-1}t_1t_2t_3)^{\frac{1}{n}}\,\big(\,(n-3)(n+5)-(n-3)(t_1+t_2+t_3)+\frac{2(2t_1t_2+2t_1t_3+2t_2t_3-t_1^2-t_2^2-t_3^2)}{t_1t_2t_3}\,\big)\,]=0.\notag
\end{align}
Recall that any four equations in the
system of the six equations $(\ref{equn_SpnEinstein01})-(\ref{equn_SpnEinstein05})$ and $(\ref{equn_SpnEinstein06})$  containing at least two of $(\ref{equn_SpnEinstein03})-(\ref{equn_SpnEinstein05})$, combining with the expansion $(\ref{equn_expansion1})$ of the Einstein metric at $x=0$, imply the other two equations.

When $\lambda_1=\lambda_2=\lambda_3$, the metric $\hat{g}$ is $\text{Sp}(k+1)\times \text{Sp}(1)$ invariant, with the additional symmetry that $\text{Sp}(1)$ acts by right multiplication. By the symmetry extension, as the Berger metric case in \cite{Li}, we have that $I_1=I_2=I_3$ for $x\in[0,1]$, and hence
\begin{align*}
y_2=y_3=y_4,
\end{align*}
for $x\in[0,1]$. Therefore, the equations $(\ref{equn_SpnEinstein01})-(\ref{equn_SpnEinstein05})$ become
\begin{align}
&\label{equn_Spn1Einstein01}y_1''-x^{-1}(1+3x^2)(1-x^2)^{-1}y_1'+\frac{1}{2n}[(y_1')^2\,+\,3(n-3)(y_2')^2]=0,\\
&\label{equn_Spn1Einstein02}y_1''-x^{-1}(2n-1+(2n+1)x^2)(1-x^2)^{-1}y_1'+\frac{1}{2}(y_1')^2\\
&+\,8(1-x^2)^{-2}[\,n(n-1)-(K^{-1}t_1^3)^{\frac{1}{n}}\,\big(\,(n-3)(n+5)-3(n-3)t_1+6t_1^{-1}\,\big)\,]=0,\notag\\
&\label{equn_Spn1Einstein03}y_2''-x^{-1}(n-1+(n+1)x^2)(1-x^2)^{-1}y_2'+\frac{1}{2}y_1'y_2'\\
&- 8 (1-x^2)^{-2} (K^{-1}t_1^3)^{\frac{1}{n}}[(n+3)t_1-n-5+ 2t_1^{-1}]=0,\notag
\end{align}
for $y_i(x)\in C^{\infty}([0,1])$ ($1\leq i \leq 2$), with the boundary condition
\begin{align}\label{equn_Spn1BV01}
t_1(0)=\frac{\lambda_1}{\lambda_4},\,\,K(1)=t_1(1)=1,\,\,y_j'(0)=y_j'(1)=0,
\end{align}
for $j=1,2$. Respectively, $(\ref{equn_SpnEinstein06})$ becomes
\begin{align}
&\label{equn_Spn1Einstein06}(y_1')^2-4nx^{-1}(1+x^2)(1-x^2)^{-1}y_1'
-\frac{3(n-3)}{(n-1)}(y_2')^2
+\frac{16n}{n-1}(1-x^2)^{-2}\,[\,n(n-1)\\
&-(K^{-1}t_1^3)^{\frac{1}{n}}\,\big(\,(n-3)(n+5)-3(n-3)t_1+6t_1^{-1}\,\big)\,]=0.\notag
\end{align}

Now we assume that two elements of $\{\lambda_1, \lambda_2, \lambda_3\}$ coincide. Without loss of generality, we assume $\lambda_2=\lambda_3$ and $\lambda_1\neq\lambda_2$. Then $\hat{g}$  is $\text{Sp}(k+1)\times \text{SU}(1)$ invariant. Then by the symmetry extension, as the Berger metric case (see Lemma 4.2 in \cite{Li}), we have that $I_2=I_3$ for $x\in[0,1]$, and hence
\begin{align*}
y_2=y_3,
\end{align*}
for $x\in[0,1]$. In this case, the equations $(\ref{equn_SpnEinstein01})-(\ref{equn_SpnEinstein05})$ become
\begin{align}
&\label{equn_Spn2Einstein01}y_1''-x^{-1}(1+3x^2)(1-x^2)^{-1}y_1'+\frac{1}{2n^2}[n(y_1')^2\,+\,((n-1)y_2'-2y_3')^2\\
&+\,2(-y_2'+(n-2)y_3')^2\,+\,(n-3)(y_2'+2y_3')^2]=0,\notag\\
&\label{equn_Spn2Einstein02}y_1''-x^{-1}(2n-1+(2n+1)x^2)(1-x^2)^{-1}y_1'+\frac{1}{2}(y_1')^2\\
&+\,8(1-x^2)^{-2}[\,n(n-1)-(K^{-1}t_1t_2^2)^{\frac{1}{n}}\,\big(\,(n-3)(n+5)-(n-3)(t_1+2t_2)\notag\\
&+\frac{2(4t_2-t_1)}{t_1t_2^2}\,\big)\,]=0,\notag\\
&\label{equn_Spn2Einstein03}y_2''-x^{-1}(n-1+(n+1)x^2)(1-x^2)^{-1}y_2'+\frac{1}{2}y_1'y_2'\\
&- 8 (1-x^2)^{-2} (K^{-1}t_1t_2^2)^{\frac{1}{n}}[(n-1)t_1+4t_2-n-5+\frac{2t_1}{t_2^2}]=0,\notag\\
&\label{equn_Spn2Einstein04}y_3''-x^{-1}(n-1+(n+1)x^2)(1-x^2)^{-1}y_3'+\frac{1}{2}y_1'y_3'\\
&-8 (1-x^2)^{-2} (K^{-1}t_1t_2^2)^{\frac{1}{n}}[(n+1)t_2+2t_1-n-5+\frac{2(2t_2-t_1)}{t_2^2}]=0,\notag
\end{align}
for $y_i(x)\in C^{\infty}([0,1])$ ($1\leq i \leq 3$) with the boundary condition
\begin{align}\label{equn_Spn2BV01}
t_i(0)=\frac{\lambda_i}{\lambda_4},\,\,K(1)=t_i(1)=1,\,\,y_j'(0)=y_j'(1)=0,
\end{align}
for $1\leq i \leq 2$ and $1\leq j \leq 3$. Respectively, $(\ref{equn_SpnEinstein06})$ becomes
\begin{align}
&\label{equn_Spn2Einstein06}(y_1')^2-4nx^{-1}(1+x^2)(1-x^2)^{-1}y_1'
-\frac{1}{n(n-1)}[((n-1)y_2'-2y_3')^2+2(-y_2'+(n-2)y_3')^2\\
&+(n-3)(y_2'+2y_3')^2]
+\frac{16n}{n-1}(1-x^2)^{-2}\,[\,n(n-1)\notag\\
&-(K^{-1}t_1t_2^2)^{\frac{1}{n}}\,\big(\,(n-3)(n+5)-(n-3)(t_1+2t_2)+\frac{2(4t_2-t_1)}{t_2^2}\,\big)\,]=0.\notag
\end{align}

 The symmetry extension approach in \cite{Li} is inspired by \cite{Wang}, see also \cite{Andersson-Dahl} and \cite{Anderson1}.

In \cite{LiQingShi}, based on the control of the relative volume growth of
geodesic balls by the Yamabe constant at the conformal infinity, we have the following curvature pinching estimates:

\begin{thm}\label{EHBoundary} (Theorem 1.6, \cite{LiQingShi}) For any $\epsilon >0$, there exists $\delta > 0$, for any conformally compact Einstein manifold $(M^{n+1}, g)$ ($n\geq 3$),
one gets
\begin{equation}\label{close-to-hyper}
|K [g] + 1| \leq \epsilon,
\end{equation}
for all sectional curvature $K$ of $g$, provided that
$$
Y(\partial M, [\hat{g}]) \geq (1-\delta) Y(\mathbb{S}^n, [g_{\mathbb{S}}]).
$$
Particularly, any conformally compact Einstein manifold with its conformal infinity of Yamabe constant sufficiently close to that of the round sphere is
necessarily negatively curved, and by \cite{WY} (see also \cite{Cai-Galloway}) if $\partial M$ is simply connected then $(M, g)$ is simply connected.
\end{thm}

It is proved in \cite{Graham} that for any smooth metric $h\in[\hat{g}]$ at the conformal infinity, there exists a unique geodesic defining function $x$ about $h$ in a neighborhood of $\partial M$. For a CCE metric of $C^2$, in \cite{CDLS} the authors showed that the following regularity result holds.

\begin{thm}\label{thm_expansion1}
Assume $\overline{M}$ is a smooth compact manifold of dimension $n+1$, $n\geq 3$, with $M$ its interior and $\partial M$ its boundary. If $g$ is a conformally compact Einstein metric of class $C^2$ on $M$ with conformal infinity $(\partial M, [\gamma])$, and $\hat{g}\in [\gamma]$ is a smooth metric on $\partial M$. Then there exists  a smooth coordinates cover of $\overline{M}$ and a smooth geodesic defining $x$ corresponding to $\hat{g}$. Under this smooth coordinates cover, the conformal compactification $\bar{g}=x^2g$ is smooth up to the boundary for $n$ odd and has the expansion
\begin{align}\label{equn_expansion1}
\bar{g}=\,dx^2+g_x=\, dx^2+\hat{g}+x^2g^{(2)}+\,(\text{even powers})\,+x^{n-1} g^{(n-1)}+ x^{n}g^{(n)}+...
\end{align}
with $g^{(k)}$ smooth symmetric $(0,2)$-tensors on $\partial M$ such that for $2k<n$, $g^{(2k)}$ can be calculated explicitly inductively using the Einstein equations and $g^{(n)}$ is a smooth trace-free nonlocal term; while for $n$ even,  $\bar{g}$ is of class $C^{n-1}$, and more precisely it is polyhomogeneous and has the expansion
\begin{align}\label{equn_expansion2}
\bar{g}=\,dx^2+g_x=\, dx^2+\hat{g}+x^2g^{(2)}+\,(\text{even powers})\,+x^n\log(x) \tilde{g}+ x^{n}g^{(n)}+...
\end{align}
with $\tilde{g}$ and $g^{(k)}$ smooth symmetric $(0,2)$-tensors on $\partial M$, such that for $2k<n$, $g^{(2k)}$ and $\tilde{g}$ can be calculated explicitly inductively using the Einstein equations, $\tilde{g}$ is trace-free and $g^{(n)}$ is a smooth nonlocal term with its trace locally determined.
\end{thm}

\section{A priori estimates and an existence result on solutions to the boundary value problem $(\ref{equn_SpnEinstein01})-(\ref{equn_SpnBV01})$}\label{section4}

In this section, for a given $\text{Sp}(k+1)$ invariant metric $\hat{g}$ on $\mathbb{S}^n$ with $n=4k+3$ $(k\geq1)$, we give an a priori estimate on the 
non-positively curved conformally compact Einstein metric with $(\mathbb{S}^n,[\hat{g}])$ as its conformal infinity. In particular, we give an estimate on the solution to the boundary value problem $(\ref{equn_SpnEinstein01})-(\ref{equn_SpnBV01})$.

By the volume comparison theorem,
\begin{align*}
K(0)=\lim_{x\to 0}\frac{\text{det}(\bar{h})}{\text{det}(\bar{h}^{\mathbb{H}^{n+1}}(x))}=\lim_{r\to +\infty}\frac{\text{det}(g_r)}{\text{det}(g_r^{\mathbb{H}^{n+1}}(r))}<1,
\end{align*}
where
\begin{align*}
g^{\mathbb{H}^{n+1}}=dr^2+g_r^{\mathbb{H}^{n+1}}(r)=x^{-2}(dx^2+\frac{(1-x^2)^2}{4}\bar{h}^{\mathbb{H}^{n+1}})
\end{align*}
is the hyperbolic metric. Moreover it is proved in \cite{LiQingShi} (see also \cite{Jin}) that
\begin{align*}
(\frac{Y(\mathbb{S}^n,[\hat{g}])}{Y(\mathbb{S}^n,[g^{\mathbb{S}^n}])})^{\frac{n}{2}}\leq K(0)=\lim_{r\to +\infty}\frac{\text{det}(g_r)}{\text{det}(g_r^{\mathbb{H}^{n+1}}(r))},
\end{align*}
where $Y(\mathbb{S}^n,[\hat{g}])$ is the Yamabe constant of $(\mathbb{S}^n,[\hat{g}])$ and $g^{\mathbb{S}^n}$ is the round sphere metric.

\begin{lem}\label{lem_monotonicitySpn01}
For the initial data $t_1(0), t_2(0),t_3(0)> 0$ different from one another, we have $y_1'(x)>0$ for $x\in(0,1)$. Moreover, if we assume that
 \begin{align}\label{ineqn_condition123}
 t_i(0)^{-1}(t_j(0)+t_k(0))>1,
 \end{align}
 for any $\{i,j,k\}=\{1, 2,3\}$, then $y_2'-y_3'$, $y_2'-y_4'$ and $y_3'-y_4'$ have no zeroes on $x\in(0,1)$. That is to say, $K$, $\frac{t_i}{t_j}$ $(i\neq j)$ are monotonic on $x\in(0,1)$.
\end{lem}
\begin{proof}
The proof is a modification of Lemma 3.1 in \cite{Li2}. Notice that $y_i$ is analytic on $x\in(0,1)$ and the zeroes of $y_i'$ ($1\leq i \leq 4$) are discrete on $x\in[0,1]$. Assume that there exists a zero of $y_1'$ on $x\in(0,1)$. Let $x_1$ be the largest zero of $y_1'$ on $x\in(0,1)$. Multiplying $x^{-1}(1-x^2)^2$ on both sides of $(\ref{equn_SpnEinstein01})$, we have
\begin{align}
&\label{equn_singularintSpn21}(x^{-1}(1-x^2)^2y_1')'+\frac{1}{2n^2}x^{-1}(1-x^2)^2[n(y_1')^2\,+\,((n-1)y_2'-y_3'-y_4')^2\\
&+\,(-y_2'+(n-1)y_3'-y_4')^2\,+\,(-y_2'-y_3'+(n-1)y_4')^2+(n-3)(y_2'+y_3'+y_4')^2]=0.\notag
\end{align}
By integrating the equation on $x\in[x_1,1]$, we have
\begin{align*}
&\frac{1}{2n^2}\int_{x_1}^1x^{-1}(1-x^2)^2[n(y_1')^2\,+\,((n-1)y_2'-y_3'-y_4')^2+\,(-y_2'+(n-1)y_3'-y_4')^2\\
&+\,(-y_2'-y_3'+(n-1)y_4')^2+(n-3)(y_2'+y_3'+y_4')^2]dx=0.
\end{align*}
Therefore, $y_1'=0$ on $x\in[x_1,1]$. Since $y_1$ is analytic, $y_1'=0$ for $x\in[0,1]$, contradicting with the fact $y_1(0)<y_1(1)$. Therefore, there is no zero of $y_1'$ on $x\in(0,1)$. Therefore, $y_1'>0$ for $x\in (0,1)$.

We multiply $x^{1-n}(1-x^2)^n$ on both sides of the equations $(\ref{equn_SpnEinstein03})$ and $(\ref{equn_SpnEinstein04})$, and take difference of the two equations obtained to have
\begin{align}\label{eqn_SinglmSpn23}
&\big(x^{1-n}(1-x^2)^n(y_2'-y_3')\big)'\,+\frac{1}{2}x^{1-n}(1-x^2)^ny_1'(y_2'-y_3')\\
&- 8 x^{1-n}(1-x^2)^{n-2} (K^{-1}t_1t_2t_3)^{\frac{1}{n}}(t_1-t_2)[(n-3)+\frac{4(t_1+t_2-t_3)}{t_1t_2t_3}]=0.\notag
\end{align}
Assume that $y_2'-y_3'$ has a zero on $x\in(0,1)$ and assume $x_{23}$ is the largest zero of $y_2'-y_3'$ on $(0,1)$. We integrate $(\ref{eqn_SinglmSpn23})$ on $x\in(x_{23}, 1)$ to have
\begin{align}
&\label{eqn_intrightSpn02}\frac{1}{2}\int_{x_{23}}^1 x^{1-n}(1-x^2)^ny_1'(y_2'-y_3') dx \\
- &8\int_{x_{23}}^1 x^{1-n}(1-x^2)^{n-2} (K^{-1}t_1t_2t_3)^{\frac{1}{n}}(t_1-t_2)[(n-3)+\frac{4(t_1+t_2-t_3)}{t_1t_2t_3}] dx=0.\notag
\end{align}
Similarly, if we assume that $x_{24}$ and $x_{34}$ are the largest zeroes of $y_2'-y_4'$ and $y_3'-y_4'$ on $(0,1)$, then we have
\begin{align}
&\label{eqn_intrightSpn03}\frac{1}{2}\int_{x_{24}}^1 x^{1-n}(1-x^2)^ny_1'(y_2'-y_4') dx\\
- &8\int_{x_{24}}^1 x^{1-n}(1-x^2)^{n-2} (K^{-1}t_1t_2t_3)^{\frac{1}{n}}(t_1-t_3)[(n-3)+\frac{4(t_1+t_3-t_2)}{t_1t_2t_3}] dx=0,\notag\\
&\label{eqn_intrightSpn04}\frac{1}{2}\int_{x_{34}}^1 x^{1-n}(1-x^2)^ny_1'(y_3'-y_4') dx \\
- &8\int_{x_{34}}^1 x^{1-n}(1-x^2)^{n-2} (K^{-1}t_1t_2t_3)^{\frac{1}{n}}(t_2-t_3)[(n-3)+\frac{4(t_2+t_3-t_1)}{t_1t_2t_3}] dx=0.\notag
\end{align}
We assume that $y_2'-y_3'$ achieves the largest zero $x_{23}\in (0,1)$ in $\{y_2'-y_3',y_2'-y_4',\,y_3'-y_4'\}$. Notice that $y_i'(1)=0$ and $y_i(1)=0$ for $i=2,3,4$. We have that $(y_i'-y_j')(t_{i-1}-t_{j-1})<0$ on $x\in(x_{23},1)$ for $2\leq i < j \leq 4$. By $(\ref{eqn_intrightSpn02})$, there exists a point $\bar{x}\in(x_{23},1)$ such that
\begin{align}\label{ineqnrightSpn23}
(n-3)+\frac{4(t_1(\bar{x})+t_2(\bar{x})-t_3(\bar{x}))}{t_1(\bar{x})t_2(\bar{x})t_3(\bar{x})}<0,
\end{align}
and therefore,
\begin{align}\label{ineqnrightSpn23-1}
\frac{t_1}{t_3}<1,\,\,\frac{t_2}{t_3}<1,
\end{align}
on the interval $x\in(x_{23}, 1)$. We {\bf claim} that there is no zero of $y_2'-y_4'$ and $y_3'-y_4'$ on $x\in (0,1)$. If that is not the case, assume $y_2'-y_4'$  achieves the largest zero $x_{24} \in (0,x_{23}]$ in $\{y_2'-y_4',y_3'-y_4'\}$. By $(\ref{eqn_intrightSpn03})$, there exists a point $\tilde{x}\in(x_{24},1)$ such that
\begin{align}\label{ineqnrightSpn24}
(n-3)+\frac{4(t_1(\tilde{x})+t_3(\tilde{x})-t_2(\tilde{x}))}{t_1(\tilde{x})t_2(\tilde{x})t_3(\tilde{x})}<0,
\end{align}
and since $y_3'-y_4'$ keeps the sign on $(x_{24}, 1)$, we have
\begin{align}\label{ineqnrightSpn24-1}
\frac{t_3}{t_2}<1,
\end{align}
on the interval $x\in(x_{24}, 1)$, contradicting with $(\ref{ineqnrightSpn23-1})$. Otherwise, assume $y_3'-y_4'$  achieves the largest zero $x_{34} \in (0,x_{23}]$ in $\{y_2'-y_4',y_3'-y_4'\}$. Then by $(\ref{eqn_intrightSpn04})$, similar argument leads to a contradiction with $(\ref{ineqnrightSpn23-1})$. That proves the {\bf claim}. Therefore,
\begin{align*}
\min\{\frac{t_i(0)}{t_3(0)}, 1\}<\frac{t_i(x)}{t_3(x)}< \max\{\frac{t_i(0)}{t_3(0)}, 1\},
\end{align*}
on $x\in(0,1)$, for $i=1,2$. This contradicts with $(\ref{ineqnrightSpn23})$. Therefore, $y_2'-y_3'$ could not achieve the largest zero on $x\in (0,1)$ among $\{y_2'-y_3',y_2'-y_4',\,y_3'-y_4'\}$.

Similarly, neither $y_2'-y_4'$ nor $y_3'-y_4'$ could achieve the largest zero on $x\in (0,1)$ among $\{y_2'-y_3',y_2'-y_4',\,y_3'-y_4'\}$. Therefore, the functions $y_2'-y_3',y_2'-y_4',\,$ and $y_3'-y_4'$ have no zeroes on $x\in(0,1)$.

\end{proof}

By $(\ref{equn_SpnEinstein06})$ and the initial value condition, for $x>0$ small, we have
\begin{align}\label{equn_Spn3y1lowerorder1}
y_1'=x^{-1}(1-x^2)^{-1}[2n(1+x^2)-\sqrt{4n^2(1+x^2)^2+\frac{1}{n(n-1)}x^2(1-x^2)^2\Psi(x)-\frac{16n}{n-1}x^2\Upsilon(x)}\,\,],
\end{align}
where
\begin{align}\label{eqn_SpnPsi}
\Psi(x)\,=&\,\,\,\,\,((n-1)y_2'-y_3'-y_4')^2+(-y_2'+(n-1)y_3'-y_4')^2\\
&+(-y_2'-y_3'+(n-1)y_4')^2+(n-3)(y_2'+y_3'+y_4')^2,\notag
\end{align}
and
\begin{align}\label{equn_SpnUpsilon}
\Upsilon(x)&=\,n(n-1)-(K^{-1}t_1t_2t_3)^{\frac{1}{n}}\,[\,(n-3)(n+5)-(n-3)(t_1+t_2+t_3)\\
&+\frac{2(2t_1t_2+2t_1t_3+2t_2t_3-t_1^2-t_2^2-t_3^2)}{t_1t_2t_3}\,].\notag
\end{align}
 Since $y_1'>0$ for $x\in(0,1)$, it is clear that
\begin{align*}
\Upsilon(x)>\frac{(1-x^2)^2}{16n^2}\Psi(x)\,\geq\, 0,
\end{align*}
and hence
\begin{align}\label{inequn_boundaryYamabeconstantSpn3}
n(n-1)>n(n-1)K^{\frac{1}{n}}>&(t_1t_2t_3)^{\frac{1}{n}}\,[\,(n-3)(n+5)-(n-3)(t_1+t_2+t_3)\\
&+\frac{2(2t_1t_2+2t_1t_3+2t_2t_3-t_1^2-t_2^2-t_3^2)}{t_1t_2t_3}\,],\notag
\end{align}
for $ x>0$ small. By continuity, this gives a lower bound of $K(0)$ using the initial data $t_i(0)$.

\begin{Remark}\label{remarkmonotonicity}
Similar proof as in Lemma \ref{lem_monotonicitySpn01} shows that $y_1'>0$ on $x\in(0,1)$ for the problems $(\ref{equn_Spn1Einstein01})-(\ref{equn_Spn1BV01})$ and $(\ref{equn_Spn2Einstein01})-(\ref{equn_Spn2BV01})$, and $(y_2'-y_3')$ has no zero on $x\in(0,1)$ for $(\ref{equn_Spn2Einstein01})-(\ref{equn_Spn2BV01})$.
\end{Remark}
By Lemma \ref{lem_monotonicitySpn01} and Remark \ref{remarkmonotonicity}, we have the following upper bound estimate of $I_4$.
\begin{cor}\label{cor_Spnboundivd}
Let $\sigma\in(0,1)$ be some given constant. For the initial data $\sigma <t_1(0), t_2(0),t_3(0)<\frac{n+5}{3}$ (without loss of generality, assume $t_1(0)\geq t_2(0)\geq t_3(0)$), assume that the inequality $(\ref{ineqn_condition123})$ holds for any $\{i,j,k\}=\{1,2,3\}$, and also there exists a constant $\tau>0$ such that
\begin{align}\label{ineqn_condition123s}
2C_{13}C_{23}+2C_{13}+2C_{23}-C_{13}^2-C_{23}^2-1>\tau,
\end{align}
for all the constants $C_{13}$ and $C_{23}$ such that
\begin{align*}
&1\leq C_{13} \leq \frac{t_1(0)}{t_3(0)},\\
&1\leq C_{23} \leq \frac{t_2(0)}{t_3(0)}.
\end{align*}
For instance, the condition $(\ref{ineqn_condition123s})$ holds if $\frac{1}{2}<\frac{t_i(0)}{t_j(0)}<2$ for $1\leq i, j\leq 3$. Then there exists a  constant $\delta=\delta(\sigma,\tau)>0$ such that
\begin{align*}
&t_i(x)\geq \delta,\\
&I_4(x)\leq \delta^{-\frac{3}{n}}
\end{align*}
for $x\in[0,1]$ and $i=1,2,3$. Moreover, we have
\begin{align}\label{ineqn_Spnupperboundt123}
t_i(x)\leq \max\{1, t_1(0), \frac{n+5}{n-1+\frac{4t_3(0)}{t_1(0)}}\},
\end{align}
for $i=1,2,3$ and $x\in[0,1]$.
\end{cor}
\begin{proof}
We start with the upper bound estimate of $t_i$. By Lemma \ref{lem_monotonicitySpn01} and Remark \ref{remarkmonotonicity}, we have that
\begin{align}
&\label{ineqn_Spnmonotonicityt123-1}\frac{t_3(0)}{t_2(0)}\leq \frac{t_3(x)}{t_2(x)}\leq 1,\\
&\label{ineqn_Spnmonotonicityt123-2}\frac{t_2(0)}{t_1(0)}\leq \frac{t_2(x)}{t_1(x)}\leq 1,
\end{align}
for $x\in[0,1]$. Assume that $t_1$ achieves the maximum value at a point $x_0\in(0,1)$. Then by $(\ref{equn_SpnEinstein03})$, we have
\begin{align*}
y_{i+1}''(x_0)=8 (1-x^2)^{-2} (K^{-1}t_1t_2t_3)^{\frac{1}{n}}[(n-1)t_1+2t_2+\,2t_3-n-5+\frac{2(t_1^2-(t_2-t_3)^2)}{t_1t_2t_3}]\big|_{x=x_0}\leq0.
\end{align*}
By the monotonicity of $\frac{t_2}{t_1}$ and $\frac{t_3}{t_2}$, we have that $t_1^2(x_0)-(t_2(x_0)-t_3(x_0))^2\geq0$, and hence,
\begin{align*}
&(n-1+\frac{4t_3(0)}{t_1(0)})t_1(x_0)-n-5\\
\leq\,&[(n-1)t_1+2t_2+\,2t_3-n-5+\frac{2(t_1^2-(t_2-t_3)^2)}{t_1t_2t_3}]\big|_{x=x_0}\leq0.
\end{align*}
Thus by the condition $(\ref{equn_SpnBV01})$, the upper bound estimate $(\ref{ineqn_Spnupperboundt123})$ is established. Therefore,
\begin{align*}
n+5-t_1-t_2-t_3>0
\end{align*}
for $x\in[0,1]$. Now by the condition $(\ref{ineqn_condition123s})$ and the monotonicity of $\frac{t_i}{t_j}$, we have
\begin{align}\label{ineqn_boundquotientoft123s}
2t_1t_2+2t_1t_3+2t_2t_3-t_1^2-t_2^2-t_3^2\geq \tau\,t_3^2>0
\end{align}
for $x\in[0,1]$. Therefore,
\begin{align*}
&4n^2(1+x^2)^2+\frac{1}{n(n-1)}x^2(1-x^2)^2\Psi(x)-\frac{16n}{n-1}x^2\Upsilon(x)\\
=&4n^2(1-x^2)^2+\frac{1}{n(n-1)}x^2(1-x^2)^2\Psi(x)-\frac{16n}{n-1}x^2\big(\Upsilon(x)-n(n-1)\big)>0
\end{align*}
holds for $x\in(0,1)$, where the functions $\Psi$ and $\Upsilon$ are defined in $(\ref{eqn_SpnPsi})$ and $(\ref{equn_SpnUpsilon})$, and hence by analyticity of $y_i$, $(\ref{equn_Spn3y1lowerorder1})$ and $(\ref{inequn_boundaryYamabeconstantSpn3})$ hold on $x\in(0,1)$.

Now we turn to the lower bound estimate. We have shown that $(\ref{ineqn_boundquotientoft123s})$ holds for $x\in[0,1]$. 
On the other hand, by $(\ref{inequn_boundaryYamabeconstantSpn3})$ for  $x\in(0,1)$, we have
\begin{align}\label{inequn_boundaryYamabeconstantSpn3-1}
n(n-1)\,>\,(t_1t_2t_3)^{\frac{1}{n}}\,[&\,(n-3)(n+5)-(n-3)(t_1+t_2+t_3)\\
&+\frac{2(2t_1t_2+2t_1t_3+2t_2t_3-t_1^2-t_2^2-t_3^2)}{t_1t_2t_3}\,],\notag
\end{align}
for $ x\in(0,1)$. By $(\ref{ineqn_Spnmonotonicityt123-1})$, $(\ref{ineqn_Spnmonotonicityt123-2})$ and $(\ref{ineqn_boundquotientoft123s})$, if $t_3\to 0_{+}$, then the right hand side of $(\ref{inequn_boundaryYamabeconstantSpn3-1})$ goes to infinity, contradicting with $(\ref{inequn_boundaryYamabeconstantSpn3-1})$. Therefore, there exists $\delta>0$ such that
\begin{align*}
t_i(x)>\delta
\end{align*}
for $x\in[0,1]$ and $i=1,2,3$. By Lemma \ref{lem_monotonicitySpn01}, $K\leq 1$, and hence
\begin{align*}
I_4=(Kt_1^{-1}t_2^{-1}t_3^{-1})^{\frac{1}{n}}\leq \delta^{-\frac{3}{n}}
\end{align*}
for $x\in[0,1]$.

This proves the corollary.

\end{proof}

Under the assumption of the corollary, we have $(\ref{ineqn_boundquotientoft123s})$, and hence, by $(\ref{equn_Spn3y1lowerorder1})$ in  $x\in(0,1)$ which is proved in Corollary \ref{cor_Spnboundivd}, we have
\begin{align}\label{equn_Spny1lowerorder21}
y_1'(x)\leq x^{-1}(1-x^2)^{-1}[2n(1+x^2)-\sqrt{4n^2(1-x^2)^2}\,\,]\leq 4n x (1-x^2)^{-1},
\end{align}
for $x\in(0,1)$. Also, by $(\ref{inequn_boundaryYamabeconstantSpn3})$ for  $x\in(0,1)$, we have
\begin{align*}
K^{\frac{1}{n}}(x)\geq K^{\frac{1}{n}}(0)\geq\frac{1}{n(n-1)}(t_1t_2t_3)^{\frac{1}{n}}\,[\,(n-3)(n+5-t_1-t_2-t_3)+\frac{2\tau t_3}{t_1t_2}\,]\big|_{x=0}.
\end{align*}

Now we give an a priori estimate of the solution away from $x=1$.

\begin{lem}\label{lem_Spnuniformests301}
Under the condition in Corollary \ref{cor_Spnboundivd}, there exists a uniform constant $C=C(\sigma,\tau)>0$ independent of the solution and the initial data $t_i(0)$ such that
\begin{align}\label{ineqn_Spny1234leftbds}
|y_i^{(k)}(x)| \leq Cx^{2-k},
\end{align}
with $y_i^{(k)}$ the $k-$th order derivative of $x$, for $k=1,2$, $1\leq i \leq 4$ and $x\in[0,\frac{3}{4}]$. The control still holds on the interval $[0, 1-\epsilon]$ for any $\epsilon>0$ small with some constant $C=C(\sigma,\tau,\epsilon)>0$. 
\end{lem}
\begin{proof}
By $(\ref{equn_Spny1lowerorder21})$ and $(\ref{equn_SpnEinstein02})$, one can easily obtain $(\ref{ineqn_Spny1234leftbds})$ for $y_1$.
Notice that by Corollary \ref{cor_Spnboundivd},
\begin{align*}
 \delta\leq t_i(x) \leq \max\{1, t_1(0), \frac{n+5}{n-1+\frac{4t_3(0)}{t_1(0)}}\},
\end{align*}
for $x\in (0,1)$ and $ 1\leq i \leq 3$, and $K(0)< K(x) <1$ for $x\in (0,1)$. By the interior estimates of the second order elliptic equations $(\ref{equn_SpnEinstein01})-(\ref{equn_SpnEinstein05})$ and the inequality $(\ref{inequn_boundaryYamabeconstantSpn3})$, there exists some constant $C=C(\sigma,\tau)>0$ independent of the initial data and the solution so that
\begin{align*}
|y_i^{(k)}(x)| \leq C,
\end{align*}
for $1\leq i \leq 4$, $0 \leq k\leq 4$ and $x \in [\frac{1}{4}, \frac{3}{4}]$. To get global estimates, we multiply $x^{1-n}(1-x^2)^nK^{\frac{1}{2}}$ on both sides of the equation $(\ref{equn_SpnEinstein03})$ and do integration on $[x,\frac{3}{4}]$ to have 
\begin{align}
&x^{1-n}(1-x^2)^nK^{\frac{1}{2}}y_2'=2^{n-1}(\frac{3}{4})^nK^{\frac{1}{2}}(\frac{1}{2})y_2'(\frac{1}{2})\\
&+ 8 \int_x^{\frac{1}{2}}s^{1-n}(1-s^2)^{n-2}K^{\frac{1}{2}} (K^{-1}t_1t_2^2)^{\frac{1}{n}}[n+5-(n-1)t_1-4t_2-\frac{2t_1}{t_2t_3}]ds,\notag
\end{align}
for $x\in(0,\frac{3}{4})$, and hence by Corollary \ref{cor_Spnboundivd}, we have
\begin{align}
\label{ineqn_Spnyd1left302}|y_2'(x)| \leq Cx,
\end{align}
for $x\in(0,\frac{3}{4})$, with some constant $C=C(\sigma,\tau)>0$ independent of the solution and the initial data. Substituting $(\ref{ineqn_Spnyd1left302})$ to $(\ref{equn_SpnEinstein03})$, we then have
\begin{align}
&|y_2''(x)| \leq C
\end{align}
for $x\in (0,\frac{3}{4})$, with $C=C(\sigma,\tau)>0$ some constant independent of the initial data and the solutions. By similar discussion, $(\ref{ineqn_Spny1234leftbds})$ holds for $i=3,4$. This completes the proof of the lemma.
\end{proof}

Recall that by the non-positivity of the sectional curvature of $g$ and the Einstein equation $(\ref{eqn_Einstein})$, we have that $|W|_g\leq T\equiv  \sqrt{n(n^2-1)}$. Using the boundedness of the Weyl tensor, we give an estimate of the solution away from $x=0$.
\begin{lem}\label{lem_SpnyiboundrightGB}
Assume that $|W|_g\leq \varepsilon$ with some constant $0<\varepsilon\leq T$. Let $\delta_0$ be any constant in $(0,1)$.  Under the condition of Corollary \ref{cor_Spnboundivd}, we have
\begin{align}
&\label{ineqn_Spnyd1right301}|y_1^{(k)}|\leq C\,\varepsilon^2(1-x^2)^{4-k},\\
&\label{ineqn_Spnyd1right302}|y_i^{(k)}|\leq C\,\varepsilon(1-x^2)^{2-k},
\end{align}
for $2\leq i \leq 4$, $k=1,2$ and $x \in [\delta_0, 1]$, with some constant $C=C(\sigma,\tau,\delta_0)>0$.
\end{lem}
\begin{proof}
By the condition of the lemma, the Weyl tensor has the bound
\begin{align}\label{ineqn_SpnWeylboundmixed}
|(g_{ii})^{-\frac{1}{2}}(g_{qq})^{-\frac{1}{2}}(g_{pp})^{-\frac{1}{2}}W_{piq0}(g)|\leq \varepsilon,
\end{align}
which is
\begin{align*}
&|(g_{ii})^{-\frac{1}{2}}(g_{qq})^{-\frac{1}{2}}(g_{pp})^{-\frac{1}{2}}W_{piq0}(g)|\\
&=|(g_{ii})^{-\frac{1}{2}}(g_{qq})^{-\frac{1}{2}}(g_{pp})^{\frac{1}{2}}[\frac{1}{2}\frac{d g_{ii}}{dr}(-(g_{ii})^{-1}+(g_{pp})^{-1}+(g_{pp})^{-1}g_{qq}(g_{ii})^{-1})\\ 
&+\frac{1}{2}\frac{dg_{pp}}{dr}((g_{pp})^{-1}-(g_{pp})^{-2}g_{ii}+(g_{pp})^{-2}g_{qq})-(g_{pp})^{-1}\frac{d g_{qq}}{dr}]|\\
&=\frac{x^2}{(1-x^2)}I_p^{\frac{1}{2}}I_q^{-\frac{1}{2}}I_i^{-\frac{1}{2}}|[I_i^{-1}\frac{d I_i}{dx}(-1+I_iI_p^{-1}+I_p^{-1}I_q)+I_p^{-1}\frac{d I_p}{dx}(1-I_iI_p^{-1}+I_p^{-1}I_q)-2(I_q^{-1}\frac{dI_q}{dx})I_p^{-1}I_q]|\\
&=2\frac{x^2}{(1-x^2)}I_q^{-\frac{1}{2}}|\frac{d}{dx}[I_i^{\frac{1}{2}}I_p^{-\frac{1}{2}}+I_i^{-\frac{1}{2}}I_p^{\frac{1}{2}}-I_i^{-\frac{1}{2}}I_p^{-\frac{1}{2}}I_q]|\leq\varepsilon,
\end{align*}
for any $\{i, p, q\}=\{1,2,3\}$; while for $1\leq q\leq 3$ and $p,i\geq 4$, the inequality becomes
\begin{align*}
\frac{2x^2}{1-x^2}I_4^{-\frac{1}{2}}|I_4^{-\frac{3}{2}}I_q^{\frac{1}{2}}I_4'-I_4^{-1}I_q^{-\frac{1}{2}}I_q'|=\frac{4x^2}{1-x^2}I_4^{-\frac{1}{2}}|\frac{d}{dx}(I_4^{-\frac{1}{2}}I_q^{\frac{1}{2}})|\leq \varepsilon.
\end{align*}
Therefore,
\begin{align*}
&|\big(-(\frac{t_i}{t_p})^{\frac{1}{2}}+(\frac{t_p}{t_i})^{\frac{1}{2}}+(\frac{t_q}{t_i})^{\frac{1}{2}}(\frac{t_q}{t_p})^{\frac{1}{2}}\big)(y_{p+1}'-y_{i+1}')+2(\frac{t_q}{t_i})^{\frac{1}{2}}(\frac{t_q}{t_p})^{\frac{1}{2}}(y_{q+1}'-y_{p+1}')|\leq \frac{1-x^2}{x^2}I_q^{\frac{1}{2}},
\end{align*}
for any $\{i, p, q\}=\{1,2,3\}$, and also it holds that
\begin{align*}
|y_{q+1}'|=2 t_q^{-\frac{1}{2}}|\frac{d}{dx}(I_4^{-\frac{1}{2}}I_q^{\frac{1}{2}})|\leq \frac{ 1-x^2}{2x^2}I_4^{\frac{1}{2}}t_q^{-\frac{1}{2}}\varepsilon,
\end{align*}
for $x\in(0,1)$ and $q=1,2,3$, and hence by Corollary \ref{cor_Spnboundivd},
\begin{align}\label{ineqn_Spny234-1}
|y_{q+1}'|\leq \frac{ 1-x^2}{2x^2}\delta^{-\frac{(3+n)}{2n}}\varepsilon
\end{align}
for $x\in(0,1)$ and $q=1,2,3$, with the constant $\delta>0$ in Corollary \ref{cor_Spnboundivd}. Set $C=\frac{1}{2\delta_0^2}\delta^{-\frac{(3+n)}{2n}}$ and hence $(\ref{ineqn_Spnyd1right302})$ holds on $x\in[\delta_0,1]$ for $k=1$ and $i=2,3,4$. Substituting the inequalities $(\ref{ineqn_Spnyd1right302})$, $(\ref{equn_Spny1lowerorder21})$, the bound of $K$ and the estimate of $t_i$ in Corollary \ref{cor_Spnboundivd} into $(\ref{equn_SpnEinstein03})-(\ref{equn_SpnEinstein05})$, we obtain immediately that there exists a constant $C=C(\sigma,\tau,\delta_0)>0$ such that $(\ref{ineqn_Spnyd1right302})$ holds for $k=2$ and $i=2,3,4$.

For $y_1'$, we multiply $x^{-1}(1-x^2)^2K^{\frac{1}{2n}}$ on both sides of $(\ref{equn_SpnEinstein01})$ and do integration on $(x, 1)$ for $\delta_0\leq x \leq 1$ so that
\begin{align*}
&(x^{-1}(1-x^2)^2K^{\frac{1}{2n}}y_1')'+ \frac{1}{2n^2}x^{-1}(1-x^2)^2K^{\frac{1}{2n}}\Psi(x)=0,\,\,\,\,\text{and}\,\,\\
&y_1'(x)=\frac{1}{2n^2}K(x)^{-\frac{1}{2n}}x(1-x^2)^{-2}\int_x^1s^{-1}(1-s^2)^2K^{\frac{1}{2n}}(s)\Psi(s)ds\\
&\leq C \,\varepsilon^2 (1-x^2)^3,
\end{align*}
with some constant $C=C(\sigma,\tau,\delta_0)>0$, where the function $\Psi$ is defined in $(\ref{eqn_SpnPsi})$. Here we have used the estimates of $t_i$ in Corollary \ref{cor_Spnboundivd}, the monotonicity of $K$ and the inequality $(\ref{ineqn_Spny234-1})$ for $q=1,2,3$. Substituting it back to $(\ref{equn_SpnEinstein01})$, we get the estimate $(\ref{ineqn_Spnyd1right301})$ for $y_1''$ with some constant $C=C(\sigma,\tau,\delta_0)>0$. This proves the lemma.
\end{proof}

Based on the estimates in Lemma \ref{lem_Spnuniformests301} and Lemma \ref{lem_SpnyiboundrightGB}, we have the compactness of the solutions to $(\ref{equn_SpnEinstein01})-(\ref{equn_SpnBV01})$. Combining with the perturbation results in Graham-Lee \cite{GL} and Lee \cite{Lee}, we can obtain an existence result of solutions to $(\ref{equn_SpnEinstein01})-(\ref{equn_SpnBV01})$ by a direct continuity argument. In particular, using the same argument of Lemma 4.4 and Theorem 1.4 in \cite{Li2}, we give the following existence result without proof. For details of the proof, one refers to \cite{Li2}.
\begin{thm}\label{thm_someSpnmetric}
Let $B_1\subseteq \mathbb{R}^{n+1}$ be the unit ball on the Euclidean space with the unit sphere $\mathbb{S}^n$ as its boundary, where $n=4k+3$ for some integer $k\geq 1$. Let $\hat{g}$ be a homogeneous metric on the boundary $\mathbb{S}^n\cong \text{Sp}(k+1)/\text{Sp}(k)$ 
so that $\hat{g}$ has the standard diagonal form
\begin{align*}
\hat{g}=\lambda_1\sigma_1^2+\lambda_2\sigma_2^2+\lambda_3\sigma_3^2+(\sigma_4^2+..+\sigma_n^2),
\end{align*}
 at a point $q\in \mathbb{S}^n$, where $\lambda_i\in [\frac{2}{3}, \frac{4}{3})$ ($1\leq i \leq 3$) are three positive constants and $\sigma_1,..,\sigma_n$ are the $1$-forms with respect to the basis vectors in $\mathfrak{p}$, in the $\text{Ad}_{\text{Sp}(k)}$-invariant splitting $\text{sp}(k+1)=\text{sp}(k)\oplus \mathfrak{p}$. For $t\in[0,1]$, we define a $1$-parameter family of metrics
   \begin{align*}
   \hat{g}^t=((1-t)\lambda_1+t)\sigma_1^2\,+\,((1-t)\lambda_2+t)\sigma_2^2\,+\,((1-t)\lambda_3+t)\sigma_3^2+(\sigma_4^2+..+\sigma_n^2).
   \end{align*}
   Then as the parameter $t$ varies from $t=1$ continuously on the interval $[0, 1]$, either it holds that there exists a conformally compact Einstein metric on $B_1$ which is non-positively curved with $(\mathbb{S}^n, [\hat{g}^t])$ as its conformal infinity for each $t\in [0, 1]$; or there exists $t_1\in(0,1)$, such that for each $t\in [t_1, 1]$ there exists a conformally compact Einstein metric $g^t$ on $B_1$ which is non-positively curved with $(\mathbb{S}^n, [\hat{g}^t])$ as its conformal infinity and there exists $p\in B_1$ such that the sectional curvature of $g^{t_1}$ is zero in some direction at $p$ and moreover, for any $\epsilon>0$ small there exists $t_2\in(t_1-\epsilon, t_1)$ such that there exists a conformally compact Einstein metric $g^{t_2}$ on $B_1$ with $(\mathbb{S}^n, [\hat{g}^{t_2}])$ as its conformal infinity and the sectional curvature of $g^{t_2}$ is positive in some direction at some point $p\in B_1$.
\end{thm}


\section{Uniqueness of the solution to the boundary value problem $(\ref{equn_Spn1Einstein01})-(\ref{equn_Spn1BV01})$}\label{sect_Spnsolutionestimates}

\begin{lem}\label{lem_monotonicitySp101}
For the initial data $t_1(0)\neq 1$, assume $(y_1,y_2)$ is a global solution of the boundary value problem $(\ref{equn_Spn1Einstein01})-(\ref{equn_Spn1BV01})$. Then we have $y_1'(x)>0$ for $x\in(0,1)$. Also, we have the inequality
\begin{align}\label{ineqn_Spn1y2upperbd}
t_1(x)\leq \max \{1,t_1(0)\}
\end{align}
for $x\in[0,1]$. Moreover, if $y_2'$ has a zero on $x\in(0,1)$, assume $x_2$ is the largest zero of $y_2'$ on $x\in(0,1)$, then for $x\in(x_2,1)$ we have
\begin{align*}
y_2'(x)>0.
\end{align*}
\end{lem}

\begin{proof}
The proof of the inequality $y_1'(x)>0$ on $x\in(0,1)$ is the same as Lemma \ref{lem_monotonicitySpn01}.

The inequality $(\ref{ineqn_Spn1y2upperbd})$ holds if $t_1$ is monotone on $x\in(0,1)$. Now we assume that $\bar{x}$ is a local maximum point of $t_1$ on $x\in(0,1)$, and hence $y_2'(\bar{x})=0$ and $y_2''(\bar{x})\leq 0$. By $(\ref{equn_Spn1Einstein03})$,
\begin{align*}
y_2''(\bar{x})=8 (1-x^2)^{-2}\, (K^{-1}t_1^3)^{\frac{1}{n}}\,(n+3-2t_1^{-1})(t_1-1)\big|_{x=\bar{x}}\,\,\leq 0.
\end{align*}
Therefore,
\begin{align}\label{ineqn_Spn1y2upperbd1}
\frac{2}{n+3} \leq t_1(\bar{x})\leq 1.
\end{align}
 This proves the inequality $(\ref{ineqn_Spn1y2upperbd})$. Similarly, at a local minimum point $\tilde{x}$ of $t_1$ on $x\in(0,1)$, we have
\begin{align*}
t_1(\tilde{x})\leq \frac{2}{n+3},\,\,\,\text{or}\,\,\,\,t_1(\tilde{x})\geq 1.
\end{align*}
Since $t_1(1)=1$ and $t_1$ has isolated zeroes, by $(\ref{ineqn_Spn1y2upperbd1})$ we have that
\begin{align}\label{ineqn_Spn1y2lowerbd}
t_1(\tilde{x})\leq \frac{2}{n+3}.
\end{align}
Let $x_2$ be the largest zero of $y_2'$ on $x\in(0,1)$. Then $(1-t_1)y_2'>0$ on $x\in(x_2,1)$. We multiply $K^{\frac{1}{2}}x^{1-n}(1-x^2)^n$ on both sides of $(\ref{equn_Spn1Einstein03})$ to have
\begin{align*}
(K^{\frac{1}{2}}x^{1-n}(1-x^2)^ny_2')'
&- 8 x^{1-n}(1-x^2)^{n-2}K^{\frac{1}{2}} (K^{-1}t_1^3)^{\frac{1}{n}}(n+3-2t_1^{-1})(t_1-1)=0,
\end{align*}
 and do integration on $x\in[x_2,1]$ to have
\begin{align}\label{ineqn_Spn1y2lowerbd1}
- 8 \int_{x_2}^1x^{1-n}(1-x^2)^{n-2}K^{\frac{1}{2}} (K^{-1}t_1^3)^{\frac{1}{n}}(n+3-2t_1^{-1})(t_1-1)dx=0.
\end{align}
If $1-t_1<0$ on $x\in(x_2,1)$, then $(n+3-2t_1^{-1})(t_1-1)>0$, contradicting with $(\ref{ineqn_Spn1y2lowerbd1})$. Therefore, $1-t_1>0$ and $y_2'>0$ on $x\in(x_2,1)$, and moreover,
\begin{align*}
t_1(x_2)< \frac{2}{n+3}.
\end{align*}
\end{proof}

By $(\ref{equn_Spn1Einstein06})$ and $(\ref{ineqn_Spn1y2upperbd})$, for $t_1(0)\leq\frac{n+5}{3}$, we have
\begin{align}
y_1'&\label{equn_Spn1y1lowerorder1}=x^{-1}(1-x^2)^{-1}[2n(1+x^2)-\sqrt{4n^2(1+x^2)^2+\frac{3(n-3)}{n-1}x^2(1-x^2)^2(y_2')^2-\frac{16n}{n-1}x^2\Upsilon(x)}\,\,]\\
&=x^{-1}(1-x^2)^{-1}[2n(1+x^2)\notag\\
&-\sqrt{4n^2(1-x^2)^2+\frac{3(n-3)}{n-1}x^2(1-x^2)^2(y_2')^2+\frac{16n}{n-1}x^2 ( K^{-1}t_1^3)^{\frac{1}{n}}\big((n-3)\big(n+5-3t_1\big)+6t_1^{-1}\big)}\,\,],\notag
\end{align}
for $x\in(0,1)$, where $\Upsilon(x)=n(n-1)-( K^{-1}t_1^3)^{\frac{1}{n}}\big((n-3)(n+5)-3(n-3)t_1+6t_1^{-1}\big)$. Since $y_1'>0$ for $x\in(0,1)$, it is clear that
\begin{align}\label{inequn_boundaryYamabeconstantSpn1}
\Upsilon(x)=n(n-1)-( K^{-1}t_1^3)^{\frac{1}{n}}\big((n-3)(n+5)-3(n-3)t_1+6t_1^{-1}\big)>\frac{3(n-3)}{n-1}x^2(1-x^2)^2(y_2')^2\,\geq\, 0,
\end{align}
for $ x\in(0,1)$.
By Lemma \ref{lem_monotonicitySp101} and $(\ref{equn_Spn1y1lowerorder1})$, we have that
\begin{align*}
y_1'&\leq x^{-1}(1-x^2)^{-1}[2n(1+x^2)
-\sqrt{4n^2(1-x^2)^2}\,\,]\\
&\leq 4n x(1-x^2)^{-1},
\end{align*}
for $x\in(0,1)$. And also by $(\ref{inequn_boundaryYamabeconstantSpn1})$, 
\begin{align*}
n(n-1)\geq n(n-1)K^{\frac{1}{n}}\geq t_1^{\frac{3}{n}}\big((n-3)(n+5)-3(n-3)t_1+6t_1^{-1}\big)
\end{align*}
for $x\in[0,1]$, and hence $t_1$ has the lower bound
\begin{align}\label{ineqn_Spn1y2lowerbd1-2}
t_1(x)>\,\big(\frac{6}{n(n-1)}\big)^{\frac{n}{n-3}},
\end{align}
for $x\in[0,1]$.

\begin{lem}\label{lem_Spn1monotonicityt1}
There exists two constants $\varepsilon>0$ and $\beta>0$ such that if $\sup_M|W|_g\leq \varepsilon$, $t_1(0)\neq0$ and $1-t_1(0)\leq \beta$, then $y_2'$ has no zero on $x\in(0,1)$. That is to say, $t_1$ is monotone on $x\in(0,1)$.
\end{lem}
\begin{proof}
Assume that
\begin{align*}
&1-t_1(0)\leq \beta,\\
&\sup_M|W|_g\leq \varepsilon,
\end{align*}
with $\beta<1-\frac{2}{n+3}$ and $\varepsilon>0$ to be determined. 
By $(\ref{ineqn_Spny1234leftbds})$, there exists $C_1>0$ independent of $\beta$ and $\varepsilon$ such that
\begin{align*}
|t_1'|\leq C_1x,
\end{align*}
for $x\in[0,\frac{1}{2}]$, 
and hence,
\begin{align*}
|t_1(x)-t_1(0)|\leq \frac{1}{2}C_1x^2
\end{align*}
for $x\in(0,\frac{1}{2})$. Therefore,
\begin{align*}
t_1(x)\geq \frac{2}{n+3}
\end{align*}
 for
 \begin{align*}
 x\leq x_0\equiv\big(\frac{2|t_1(0)-\frac{2}{n+3}|}{C_1}\big)^{\frac{1}{2}}.
 \end{align*}
Using the bound $\sup_M|W|_g\leq \varepsilon$ and $(\ref{ineqn_Spnyd1right302})$, we have that
\begin{align*}
|y_2(x_0)|\leq C_2 \varepsilon \int_{x_0}^1(1-s^2)ds= C_2(1-x_0-\frac{1}{3}+\frac{1}{3}x_0^3)\varepsilon.
\end{align*}
for some constant $C_2=C_2(x_0)>0$, and hence, if
\begin{align*}
\varepsilon< -\frac{1}{C_2(1-x_0-\frac{1}{3}+\frac{1}{3}x_0^3)}\ln(\frac{2}{n+3}),
\end{align*}
using the estimate $(\ref{ineqn_Spn1y2lowerbd})$ of the local minimum of $t_1$, we have that there is no minimum point of $t_1$ on $x\in(0,1)$. Therefore, by the boundary value condition $(\ref{equn_Spn1BV01})$, $y_2'$ has no zero on $x\in(0,1)$. This completes the proof of the lemma.

\end{proof}

We assume that the boundary value problem $(\ref{equn_Spn1Einstein01})-(\ref{equn_Spn1BV01})$ admits two solutions $(y_{11}, y_{12})$ and $(y_{21}, y_{22})$ with $y_{11}=\log(K_1),\,y_{12}=\log(t_{11}),\,y_{21}=\log(K_2)\,$ and $y_{22}=\log(t_{21})$, for $t_1(0)\neq 1$ close to $1$. Let $z_i=y_{1i}-y_{2i}$, for $i=1,2$. By the same argument in Lemma 5.3 in \cite{Li}, we have 
\begin{lem}\label{lem_zeroz1z2}
Under the condition in Lemma \ref{lem_Spn1monotonicityt1}, for any two zeroes $0<x_1<x_2\leq 1$ of $z_1'$ so that there is no zero of $z_1'$ on the interval $x\in(x_1,x_2)$, there exists a point $x_3\in(x_1,x_2)$ so that
\begin{align}\label{inequn_signz201}
(y_{12}'+y_{22}')z_1'z_2'\big|_{x=x_3}<0.
\end{align}
Also, for any zero $0<x_2\leq 1$ of $z_1'$, there exists $\varepsilon_1>0$ so that for any $x_2-\varepsilon_1 < x <x_2$, we have
\begin{align}\label{inequn_signz202}
(y_{12}'(x)+y_{22}'(x))z_1'(x)z_2'(x)>0.
\end{align}
\end{lem}

It is clear that the function
\begin{align*}
t_1^{\frac{3}{n}}[(n+3)t_1-n-5+ 2t_1^{-1}]
\end{align*}
of $t_1$ is increasing on the interval $t_1\in(t_1^0, +\infty)$ with
\begin{align*}
t_1^0\equiv \frac{3n+15+\sqrt{(3n+15)^2+4(2n-6)(n+3)^2}}{2(n+3)^2}<1.
\end{align*}
 By the same proof of Theorem 5.4 in \cite{Li}, with the integrating factor $x^{-2}(1-x^2)^3$ in $(5.13)$ in \cite{Li} replaced by $x^{1-n}(1-x^2)^n$, we have the following uniqueness lemma for the boundary value problem $(\ref{equn_Spn1Einstein01})-(\ref{equn_Spn1BV01})$. ( Notice that Theorem 5.2 in \cite{Li} is not necessary for the uniqueness argument. Also, for $K_1(0)=K_2(0)$, by the mean value theorem, there exists a zero of $z_1'$ in $x\in(0,1)$, and Theorem 5.4 in \cite{Li} covers this case.)
\begin{lem}\label{lem_Spn1metricStability}
The solution to the boundary value problem $(\ref{equn_Spn1Einstein01})-(\ref{equn_Spn1BV01})$, with $y_1$ and $y_2$ monotone on $x\in(0,1)$ and $t_1(0)>t_1^0$, must be unique if it exists.
\end{lem}

Now we prove the global uniqueness of the conformally compact Einstein metric with $\text{Sp}(k+1)\times\text{Sp}(1)$ invariant conformal infinity.

\begin{thm}\label{thm_someSpn1metric}
Let $\hat{g}$ be a homogeneous metric on $\mathbb{S}^n\cong \text{Sp}(k+1)\times \text{Sp}(1)/\text{Sp}(k)\times\text{Sp}(1)$ with $n=4k+3$ for $k\geq 1$ so that $\hat{g}$ has the standard diagonal form
\begin{align}\label{eqn_Spn1standardmetricform}
\hat{g}=\lambda_1(\sigma_1^2+\sigma_2^2+\sigma_3^2)+\lambda_2(\sigma_4^2+..+\sigma_n^2),
\end{align}
at a point $p\in \mathbb{S}^n$, where $\lambda_1$ and $\lambda_2$ are two positive constants and $\sigma_1,..,\sigma_n$ are the $1$-forms with respect to the basis vectors in $\mathfrak{p}$, in the $\text{Ad}_{Sp(k)}$-invariant splitting $sp(k+1)=sp(k)\oplus \mathfrak{p}$. Assume that $\frac{\lambda_1}{\lambda_2}$ is close enough to $1$, then up to isometry the conformally compact Einstein metric filled in is unique and it is the perturbation metric in \cite{GL} on the $(n+1)$-ball $B_1(0)$ with $(\mathbb{S}^n, [\hat{g}])$ as its conformal infinity. 
\end{thm}
\begin{proof}[Proof of Theorem \ref{thm_someSpn1metric}]
For the case $t_1(0)=1$ i.e., $\lambda_1=\lambda_2$ so that the conformal infinity is the round sphere, the theorem has been proved in \cite{Andersson-Dahl}\cite{Q}\cite{DJ}\cite{LiQingShi}.

Now we assume that $\lambda_1\neq \lambda_2.\,$ By Theorem \ref{EHBoundary} and the continuity of the Yamabe constant, for $\frac{\lambda_1}{\lambda_2}$ close enough to $1$, the conformally compact Einstein manifold $(M, g)$ filled in is non-positively curved and simply connected and the closure $\overline{M}=M\bigcup \partial M$ of $M$ is diffeomorphic to the unit ball $\overline{B}_1\subset \mathbb{R}^{n+1}$. Moreover, the sectional curvature of $g$ is close to $-1$ and hence $\sup_M|W|_g$ is small. Thus, the condition in Lemma \ref{lem_Spn1monotonicityt1} is satisfied.

Pick up a point $q\in \partial M=\mathbb{S}^n$. Let $x$ be the geodesic defining function about $C\hat{g}$ with $C>0$ some constant so that $x=e^{-r}$ with $r$ the distance function on $(M,g)$ to the center of gravity $p_0\in M$, see Theorem 3.6 in \cite{Li}.  Under the polar coordinate $(x, \theta)$ with $0\leq x\leq 1$ and $\theta=0$ along the geodesic $\gamma$ connecting $q$ and $p_0$, by \cite{Li2} and Section \ref{Sect:preliminary} we have that  the Einstein equations with prescribed conformal infinity $(\mathbb{S}^n, [\hat{g}])$ with $\hat{g}$ the homogeneous metric in $(\ref{eqn_Spn1standardmetricform})$, is equivalent to the boundary value problem $(\ref{equn_Spn1Einstein01})-(\ref{equn_Spn1BV01})$ along the geodesic $\gamma$ provided that the solution has non-positive sectional curvature. Moreover, since $t_1(0)$ is close to $1$ and the condition in Lemma \ref{lem_Spn1monotonicityt1} is satisfied, by Lemma \ref{lem_Spn1monotonicityt1}, $y_1$ and $y_2$ are monotone on $x\in(0,1)$. Moreover, Lemma \ref{lem_zeroz1z2} holds. Then by Lemma \ref{lem_Spn1metricStability}, up to isometries, the CCE metric is unique.
\end{proof}

\section{On the monotonicity of the solution to the boundary value problem $(\ref{equn_Spn2Einstein01})-(\ref{equn_Spn2BV01})$}\label{section5}

\begin{lem}\label{lem_monotonicitySp201}
Let $(y_1,y_2,y_3)$ be a global solution of the boundary value problem $(\ref{equn_Spn2Einstein01})-(\ref{equn_Spn2BV01})$. Then we have $y_1'(x)>0$ for $x\in(0,1)$. Define the function $t_1^*(x)$ as
\begin{align}\label{equn_Spn2t1star01}
t_1^*(x)\equiv \frac{(n+5)t_2^2(x)-4t_2^3(x)}{(n-1)t_2^2(x)+2},
\end{align}
for $x\in[0,1]$. Then, at any local maximum point $p_1$ of $t_1$ on $x\in(0,1)$, we have
\begin{align}\label{ineqn_Spn2t1ubds1}
t_1(p_1)\leq t_1^*(p_1),
\end{align}
while at any local minimum point $q_1$ of $t_1$ on $(0,1)$, we have
\begin{align}\label{ineqn_Spn2t1lbds1}
t_1(q_1)\geq t_1^*(q_1).
\end{align}
Moreover,
at any local maximum point $p_2$ of $t_2$ on $(0,1)$, we have
\begin{align}\label{ineqn_Spn2t2ubds1}
t_2(p_2)\leq 1,
\end{align}
while at any local minimum point $q_2$ of $t_2$ on $(0,1)$, we have
\begin{align}\label{ineqn_Spn2t2lbds1}
t_2(q_2)< \frac{4}{n+1},\,\,\,\text{and}\,\,\,t_1(q_2)\leq \frac{4t_2-(n+1)t_2^2}{2t_2+2}.
\end{align}
\end{lem}

\begin{proof}
The proof of the inequality $y_1'(x)>0$ on $x\in(0,1)$ is the same as Lemma \ref{lem_monotonicitySpn01}.

It is clear that $t_1^*(x)\leq 1$ for $x\in[0,1]$ by the definition of $t_1^{*}(x)$, with $t_1^*(x)=1$ if and only if $t_2(x)=1$. At any local maximum (resp. minimum) point $p_1$ (resp. $q_1$) of $t_1$ on $(0,1)$, we have $y_2'=0$ and $y_2''(p_1)\leq 0$ (resp. $y_2''(q_1) \geq 0$), which immediately yields $(\ref{ineqn_Spn2t1ubds1})$ and $(\ref{ineqn_Spn2t1lbds1})$, by $(\ref{equn_Spn2Einstein03})$.

By $(\ref{equn_Spn2Einstein04})$, at a local maximum point $p_2$ of $t_2$ on $(0,1)$, we have
\begin{align*}
y_3''(p_2)=8 (1-x^2)^{-2} (K^{-1}t_1t_2^2)^{\frac{1}{n}}t_2^{-2}(t_2-1)[2(1+t_2)t_1+ (n+1)t_2^2-4t_2]\big|_{x=p_2}\leq 0,
\end{align*}
and hence $t_2(p_2)\leq 1$; while at a local minimum point $q_2$ of $t_2$ on $(0,1)$, we have
\begin{align}\label{ineqn_Spn2minimumt21}
(t_2-1)[2(1+t_2)t_1+ (n+1)t_2^2-4t_2]\big|_{x=q_2}\geq 0.
\end{align}
Since $t_2(1)=1$, by $(\ref{ineqn_Spn2t2ubds1})$, we have that $t_2(q_2)< 1$. Therefore, by $(\ref{ineqn_Spn2minimumt21})$, we obtain $(\ref{ineqn_Spn2t2lbds1})$.

\end{proof}

Now we show that for $\sup_M |W|_g$ small and $t_2(0)$ not too small, $t_2$ is monotone on $x\in(0,1)$.

\begin{lem}\label{lem_Spn2monotonicityt2}
There exists two constants $\varepsilon>0$ and $\beta>0$ such that if $\sup_M|W|_g\leq \varepsilon$, $t_2(0)\neq 1$ and $1-t_2(0)\leq \beta$, then $y_3'$ has no zero on $x\in(0,1)$. That is to say, $t_2$ is monotone on $x\in(0,1)$. Moreover, $t_1^*(x)$ (see $(\ref{equn_Spn2t1star01})$) is increasing and there is no local maximum point of $t_1$ on $x\in(0,1)$.
\end{lem}
\begin{proof}
Assume that
\begin{align*}
&1-t_2(0)\leq \beta,\\
&\sup_M|W|_g\leq \varepsilon,
\end{align*}
with $\beta<1-\frac{4}{n+1}$ and $\varepsilon>0$ to be determined. 
By $(\ref{ineqn_Spny1234leftbds})$, there exists $C_1>0$ independent of $\beta$ and $\varepsilon$ such that
\begin{align*}
|t_2'|\leq C_1x,
\end{align*}
for $x\in[0,\frac{1}{2}]$, and hence,
\begin{align*}
|t_2(x)-t_2(0)|\leq \frac{1}{2}C_1x^2
\end{align*}
for $x\in(0,\frac{1}{2})$. Therefore,
\begin{align*}
t_2(x)\geq \frac{4}{n+1}
\end{align*}
 for
 \begin{align*}
 x\leq x_0\equiv\big(\frac{2|t_2(0)-\frac{4}{n+1}|}{C_1}\big)^{\frac{1}{2}}.
 \end{align*}
Using the bound $\sup_M|W|_g\leq \varepsilon$ and $(\ref{ineqn_Spnyd1right302})$, we have that
\begin{align*}
|y_3(x_0)|\leq C_2 \varepsilon \int_{x_0}^1(1-s^2)ds= C_2(1-x_0-\frac{1}{3}+\frac{1}{3}x_0^3)\varepsilon,
\end{align*}
for some constant $C_2=C_2(x_0)>0$, and hence, if
\begin{align*}
\varepsilon< -\frac{1}{C_2(1-x_0-\frac{1}{3}+\frac{1}{3}x_0^3)}\ln(\frac{4}{n+1}),
\end{align*}
using the estimate $(\ref{ineqn_Spn2t2lbds1})$ of the local minimum of $t_2$, we have that there is no minimum point of $t_2$ on $x\in(0,1)$. Therefore, by the boundary value condition $(\ref{equn_Spn1BV01})$ and $(\ref{ineqn_Spn2t2ubds1})$, $y_3'$ has no zero on $x\in(0,1)$.

Now we turn to the monotonicity of $y_2$. By the definition of $t_1^*$ in $(\ref{equn_Spn2t1star01})$, it is easy to check that for $t_2\in(\frac{2}{n+3},1)$, we have $t_1^*\in(t_2,1)$, while for $t_2>1$, we have $t_1^*<1$. Moreover, $t_1^*(x)$ is increasing when $t_2(x)<1$ and $t_2(x)$ is increasing, while $t_1^*(x)$ is increasing when $t_2(x)>1$ and $t_2(x)$ is decreasing. Hence by the choice of $\varepsilon$ and $\beta$ above, we obtain that $t_1^*(x)$ is increasing on $x\in(0,1)$.

Assume that there exists a local maximum point $p_1$ of $t_1$ on $x\in(0,1)$. Then by $(\ref{ineqn_Spn2t1ubds1})$, we have that
\begin{align*}
t_1(p_1)\leq t_1^*(p_1)<1.
\end{align*}
Since $t_1(1)=1$, there must be a local minimum $q_1$ of $t_1$ next to $p_1$ on the interval $x\in(p_1,1)$. Therefore, $t_1(q_1)<t_1(p_1)$ and by $(\ref{ineqn_Spn2t1lbds1})$, we have $t_1(q_1)\geq t_1^*(q_1)$, contradicting with the fact that $t_2^*(x)$ is increasing on $x\in(0,1)$. Therefore, there exists no local maximum point of $t_1$ on $x\in(0,1)$.

This completes the proof of the lemma.

\end{proof}

\begin{lem}\label{lem_Spn2monotonicityt1t22}
Under the condition in Lemma \ref{lem_Spn2monotonicityt2}, for $t_1(0)\leq t_1^*(0)$ with $t_1^*$ defined in $(\ref{equn_Spn2t1star01})$, we have that $y_2'$ has no zero and $t_1$ is increasing on $x\in(0,1)$; while for $t_1^*(0)<t_1(0)\leq 1$, $t_1$ is decreasing for $x>0$ small till the unique minimum point $q_1$ of $t_1$ on $x\in(0,1)$, and then $t_1$ keeps increasing on $x\in(q_1, 1)$; while for $t_1(0)>1$, $t_1$ is decreasing for $x>0$ small with at most one minimum point $q_1$ on $x\in(0,1)$. Moreover,
\begin{align}\label{ineqn_Spn2t1lbds2}
t_1(x)\geq\min\{t_1(0), t_1^*(0)\}
\end{align}
for $x\in[0,1]$.
\end{lem}

\begin{proof}
By Lemma \ref{lem_Spn2monotonicityt2}, we have that $t_1^*(x)<1$ is increasing on $x\in(0,1)$, and $t_1$ has no local maximum point on $x\in(0,1)$, and therefore, there exists at most one local minimum point of $t_1$ on $x\in(0,1)$. We only need to consider the sign of $y_2'$ for $x>0$ small. By $(\ref{equn_Spn2Einstein03})$, we have
\begin{align*}
y_2''(0)=\,- \frac{8}{n-2}(K^{-1}t_1t_2^2)^{\frac{1}{n}}\big((n-1)+2t_2^{-2}(0)\big)[t_1(0)-t_1^*(0)].
\end{align*}
Therefore, $y_2'(x)<0$ for $x>0$ small when $t_1(0)>t_1^*(0)$, while $y_2'(x)>0$ for $x>0$ small when $t_1(0)<t_1^*(0)$.

For $t_1(0)=t_1^*(0)$, if $y_1'(x)<0$ for $x>0$ small, then $t_1(x)$ is decreasing till a local minimum $q_1\in(0,1)$ of $t_1$ and hence, $t_1(q_1)<t_1(0)$. By Lemma \ref{lem_Spn2monotonicityt2}, $t_1^*$ is increasing on $x\in(0,1)$, which is a contradiction with the inequality $(\ref{ineqn_Spn2t1lbds1})$. Therefore, $t_1$ keeps increasing on $x\in(0,1)$ when $t_1(0)=t_1^*(0)$.

Since $t_1\geq t_1^*$ at the minimum point of $t_1$ and $t_1^*$ is increasing on $x\in(0,1)$, we obtain the inequality $(\ref{ineqn_Spn2t1lbds2})$.

Recall that $(\ref{inequn_boundaryYamabeconstantSpn3})$ holds on $x\in(0,1)$ by Corollary \ref{cor_Spnboundivd}. Therefore, there exists a constant $C=C(\beta)>0$ with $\beta$ in Lemma \ref{lem_Spn2monotonicityt2}, such that
\begin{align*}
|y_i(x)|\leq C(|1-t_1(0)|+|1-t_2(0)|),
\end{align*}
for $x\in[0,1]$ and $i=1,2,3$, and hence for the problem $(\ref{equn_Spn2Einstein01})-(\ref{equn_Spn2BV01})$, the constant $C$ in $(\ref{ineqn_Spny1234leftbds})$ on the interval $x\in(0,1-\epsilon)$ can be replaced by $C(|1-t_1(0)|+|1-t_2(0)|)$ with some constant $C=C(\beta,\tau,\epsilon)>0$.

\end{proof}

\section{Uniqueness of the solution to the boundary value problem $(\ref{equn_SpnEinstein01})-(\ref{equn_SpnBV01})$}\label{section6}

Based on the a priori estimates in Section \ref{sect_Spnsolutionestimates}, we will follow the approach of \cite{Li2} to prove the uniqueness of the conformally compact Einstein metrics filling in for the given conformal infinity $(\mathbb{S}^n, [\hat{g}])$ with $n=4k+3$ and $\hat{g}$ an $\text{Sp}(k+1)$-invariant metric close enough to the round sphere metric.

\begin{proof}[Proof of Theorem \ref{thm_uniquenessSpn}]
For the case $\lambda_i=\lambda_j$ for $1\leq i, j\leq 4$, so that the conformal infinity is the round sphere, the theorem has been proved in \cite{Andersson-Dahl}\cite{Q}\cite{DJ}\cite{LiQingShi}. Otherwise, by Theorem \ref{EHBoundary} and the continuity of the Yamabe constant at a positive scalar curvature metric, for $\frac{\lambda_i}{\lambda_j}$ close to $1$, the conformally compact Einstein manifold $(M, g)$ filling in is negatively curved and simply connected with the closure $\overline{M}$ diffeomorphic to the unit ball $\overline{B}_1\subset \mathbb{R}^{n+1}$. By the discussion in \cite{Li2}, we have that the uniqueness of the conformally compact Einstein metric with $(\mathbb{S}^n, [\hat{g}])$ as its conformal infinity is equivalent to the uniqueness of the solution to the boundary value problem $(\ref{equn_SpnEinstein01})-(\ref{equn_SpnBV01})$.

We assume that the boundary value problem $(\ref{equn_SpnEinstein01})-(\ref{equn_SpnBV01})$ admits two solutions $(y_{11}, y_{12}, y_{13}, y_{14})$ and $(y_{21}, y_{22}, y_{23}, y_{24})$ with $y_{11}=\log(K_1),\,y_{21}=\log(K_2)\,$ and $y_{i(j+1)}=\log(t_{ij})$ for $i=1,2$ and $j\in\{1,2,3\}$, where
\begin{align*}
|1-t_1(0)|+|1-t_2(0)|+|1-t_3(0)|\neq 0,
\end{align*}
and $t_j(0)$ is close to $1$ for $j=1,2,3$. Denote $z_j=y_{1j}-y_{2j}$ for $1\leq j\leq 4$. In general, one is not able to prove that $z_j$ is monotone on $x\in(0,1)$ as the Berger metric case in \cite{Li}. Instead, one turns to the control of the total variation of $z_j$. We will follow the approach of \cite{Li2} to show uniqueness of the solution to $(\ref{equn_SpnEinstein01})-(\ref{equn_SpnBV01})$ for $t_i(0)$ close enough to $1$.

Let $0<\tau_1 <\frac{1}{3}$ and $0<\varepsilon \leq T \equiv\sqrt{n(n^2-1)}$ be two small numbers to be chosen.

We assume that
\begin{align}\label{ineqn_initialdatabbds}
|1-t_1(0)|+|1-t_2(0)|+|1-t_3(0)|\leq \tau_1,
\end{align}
and
\begin{align}
\sup_M|W|_g\leq \varepsilon.
\end{align}
The conditions in Corollary \ref{cor_Spnboundivd}, Lemma \ref{lem_Spnuniformests301} and Lemma \ref{lem_SpnyiboundrightGB} hold by the assumption $(\ref{ineqn_initialdatabbds})$. Therefore, there exists a constant $C_1>0$ independent of $\tau_1$ and $\varepsilon$ such that
\begin{align}\label{ineqn_Spnleftboundy11}
|y_{ij}'(x)|\leq C_1x
\end{align}
for $x\in[0,\frac{3}{4}]$, with $i=1,2$ and $1\leq j\leq 4$; and for any $\delta_0\in(0,1)$, there exists a constant $C_2=C_2(\delta_0)>0$ such that
\begin{align}\label{ineqn_Spnrightboundy11}
|y_{i1}'(x)|\leq C_2\varepsilon^2(1-x)^3,\,\,\, \text{and}\,\,\,|y_{ij}'(x)|\leq C_2\varepsilon (1-x),
\end{align}
for $x\in(\delta_0, 1)$, with $i=1,2$ and $2\leq j\leq 4$. Therefore, let $\delta_0=\frac{\tau_1}{6C_1}$ such that for $x\in(0, \delta_0]$, we have $|y_{ij}'(x)|\leq \frac{\tau_1}{6}$ and by integration,
\begin{align}\label{ineqn_Spnleftboundy012}
|1-t_{i1}(x)|+|1-t_{i2}(x)|+|1-t_{i3}(x)|\leq 2\tau_1.
\end{align}
Moreover, by $(\ref{inequn_boundaryYamabeconstantSpn3})$ and the monotonicity of $K_i$, there exists $C_3>0$ independent of $\tau_1$ such that
\begin{align}\label{ineqn_Spnleftboundy013}
|1-K_i(x)|\leq C_3 \tau_1,
\end{align}
for $x\in[0,1]$ and $i=1,2$. Also, for $x\in[\delta_0,1)$, by integrating $(\ref{ineqn_Spnrightboundy11})$ on $(x, 1)$, we have
\begin{align}\label{ineqn_Spnrightbound012}
|K_i(x)-1|\leq \frac{1}{2}C_2\varepsilon^2(1-x)^4,\,\,\,\text{and}\,\,\,|1-t_{ij}(x)|\leq C_2\varepsilon (1-x)^2,
\end{align}
for $i=1,2$ and $1\leq j \leq 3$.

Recall that if two solutions share the same conformal infinity $[\hat{g}]$ and the same non-local term $g^{(n)}$ in the expansion, then they coincide by \cite{Biquard2}. Also, when the two solutions are distinct, by the Einstein equations $(\ref{equn_SpnEinstein01})-(\ref{equn_SpnEinstein05})$ and the expansion $(\ref{thm_expansion1})$, the zeroes of $z_i$ are discrete on $x\in[0,1]$ for $1\leq i \leq 4$.

 For $i=2, 3, 4$, we define the domain
\begin{align*}
D_i^-=\{x\in(0,1] \big| z_i(x)\leq 0\}.
 \end{align*}
 Let $b_{i1}^-<...<b_{im_i}^-$ be the list of the local minimum points of $z_i$ on $D_i^-$, and we pick up all the (maximal) non-increasing intervals of $z_i$ on $\overline{D}_i^-$ (the closure of $D_i^-$):
 \begin{align*}
 [a_{i1}^-,b_{i1}^-]\bigcup [a_{i2}^-,b_{i2}^-] \bigcup...\bigcup [a_{im_i}^-,b_{im_i}^-]
 \end{align*}
 such that $a_{i1}^-<b_{i1}^-<a_{i2}^-<...<a_{im_i}^-<b_{im_i}^-$ with $m_i$ some integer; while on the domain
 \begin{align*}
 D_i^+=\{x\in(0,1] \big| z_i(x)\geq 0\},
  \end{align*}
  let $b_{i1}^+<...<b_{in_i}^+$ be all the local maximum points of $z_i$ on $D_i^+$ with $n_i$ some integer, and we pick up all the (maximal) non-decreasing intervals of $z_i$ on $\overline{D}_i^+$ (the closure of $D_i^+$):
  \begin{align*}
  [a_{i1}^+,b_{i1}^+]\bigcup [a_{i2}^+,b_{i2}^+] \bigcup...\bigcup [a_{in_i}^+,b_{in_i}^+]
  \end{align*}
such that $a_{i1}^+<b_{i1}^+<a_{i2}^+<...<a_{in_i}^+<b_{in_i}^+$.  Since $z_i\in C^{\infty}([0,1])$ and $z_i'$ has finitely many zeroes, the function $z_i$ is of bounded variation on $x\in[0,1]$. Given an interval $[a,b] \subseteq [0,1]$, we denote $V_a^b(z_i)$ the total variation of $z_i$ on $x \in [a,b]$, and we denote $V(z_i)$ the total variation of $z_i$ on $x\in[0,1]$. When each $t_i(0)$ is close to $1$, we will show by the Einstein equations that the total variation of each $z_i$ is controlled by the linear combination of those of the other three with small coefficients, which implies that the total variation of each $z_i$ vanishes and the two solutions coincide.

 Recall that $z_i(0)=z_i(1)=0$ for $i=2, 3, 4$. By the mean value theorem, there exists at least one zero of $z_i'$ on $x\in(0,1)$. Also for $1 \leq j \leq m_i$, we have either $z_i(a_{ij}^-)=0$, or $z_i(a_{ij}^-)\leq 0$ with $a_{ij}^-$ a local maximum of $z_i$ on $[0,1)$ and $z_i'(a_{ij}^-)=0$. Similarly, for $1 \leq j \leq n_i$, it holds that either $z_i(a_{ij}^+)=0$, or $z_i(a_{ij}^+)\geq 0$ with $a_{ij}^+$ a local minimum of $z_i$ on $[0,1)$ and $z_i'(a_{ij}^+)=0$. Hence we have that for $i=2,3,4$,
 \begin{align*}
\frac{1}{2}V(z_i)&=\displaystyle\sum_{j=1}^{m_i}V_{a_{ij}^-}^{b_{ij}^-}(z_i)+\displaystyle\sum_{j=1}^{n_i}V_{a_{ij}^+}^{b_{ij}^+}(z_i)\\
&=\displaystyle\sum_{j=1}^{m_i}|z_i(b_{ij}^-) - z_i( a_{ij}^-)|\,+\,\displaystyle\sum_{j=1}^{n_i}|z_i(b_{ij}^+) - z_i(a_{ij}^+)|.
\end{align*}
We multiply $x(1-x^2)$ on both sides of $(\ref{equn_SpnEinstein03})$ and obtain
\begin{align}\label{equn_smoothSpnEinstein03}
&(x(1-x^2)y_2')'-(n+(n-2)x^2)y_2'+\frac{1}{2}x(1-x^2)y_1'y_2'\\
&- 8 x(1-x^2)^{-1} (K^{-1}t_1t_2t_3)^{\frac{1}{n}}[(n-1)t_1+2t_2+2t_3-n-5+\frac{2(t_1^2-(t_2-t_3)^2)}{t_1t_2t_3}]=0.\notag
\end{align}
Now we substitute the two solutions to $(\ref{equn_smoothSpnEinstein03})$ and take difference to obtain
\begin{align}\label{equn_Spny12y22z2}
0=&\,(x(1-x^2)z_2')'-(n+(n-2)x^2)z_2'+\frac{1}{2}x(1-x^2)(y_{11}'z_2'+z_1'y_{22}')- 8 x(1-x^2)^{-1}\Phi(x)\\
&- 8 x(1-x^2)^{-1}(K_2^{-1}t_{21}t_{22}t_{23})^{\frac{1}{n}}(n-1+2t_{22}^{-1}t_{23}^{-1})(t_{11}-t_{21}),\notag
\end{align}
with
\begin{align*}
\Phi(x)=&[(K_1^{-1}t_{11}t_{12}t_{13})^{\frac{1}{n}}-(K_2^{-1}t_{21}t_{22}t_{23})^{\frac{1}{n}}]\,[(n-1)t_{11}+2t_{12}+2t_{13}-n-5+\frac{2(t_{11}^2-(t_{12}-t_{13})^2)}{t_{11}t_{12}t_{13}}]\\
&+ 2(K_2^{-1}t_{21}t_{22}t_{23})^{\frac{1}{n}}[t_{12}+t_{13}-t_{22}-t_{23}+t_{11}(\frac{1}{t_{12}t_{13}}-\frac{1}{t_{22}t_{23}})+t_{11}^{-1}\big(\frac{1}{t_{22}}+\frac{1}{t_{23}}-\frac{1}{t_{12}}-\frac{1}{t_{13}}\big)]\\
&+2(K_2^{-1}t_{21}t_{22}t_{23})^{\frac{1}{n}}\frac{(t_{22}-t_{23})^2}{t_{22}t_{23}}(-\frac{1}{t_{11}}+\frac{1}{t_{21}}).
\end{align*}
By $(\ref{ineqn_Spnleftboundy012})-(\ref{ineqn_Spnrightbound012})$, there exist a constant $C_4=C_4(\delta_0)>0$ depending on $\tau_1$ and a constant $C_5>0$ independent of $\tau_1$ and $\varepsilon$ such that
\begin{align}
x(1-x^2)^{-1}|\Phi(x)|\leq C_5x(\tau_1+C_4 \varepsilon)\,\big(\sum_{i=1}^4|z_i(x)|\big),
\end{align}
for $x\in[0,1]$. Integrating the equation $(\ref{equn_Spny12y22z2})$ on $[a_{2j}^{\pm}, b_{2j}^{\pm}]$ to have
\begin{align}\label{ineqn_Spnintz21}
&(x-x^3)z_2'(x)\big|_{x=a_{2j}^{\pm}}+n(z_2(b_{2j}^{\pm})-z_2(a_{2j}^{\pm}))+(n-2)\int_{a_{2j}^{\pm}}^{b_{2j}^{\pm}}x^2z_2'(x)dx\\
=&\int_{a_{2j}^{\pm}}^{b_{2j}^{\pm}}\big[ \frac{1}{2}(x-x^3)(y_{11}'z_2'+z_1'y_{22}')+O(1)\,x\,(\tau_1+C_4 \varepsilon)\,\big(\sum_{i=1}^4|z_i(x)|\big)\notag\\
&- 8 x(1-x^2)^{-1}(K_2^{-1}t_{21}t_{22}t_{23})^{\frac{1}{n}}(n-1+2t_{22}^{-1}t_{23}^{-1})(t_{11}-t_{21})\,\big]\,dx,\notag
\end{align}
for any $j$, with $O(1)$ uniformly bounded, independent of $\tau_1$ and $\varepsilon$. It is clear that the three terms on the left hand side of the equation have the same sign. Notice that 
\begin{align*}
n-1+2t_{22}^{-1}t_{23}^{-1}>0
\end{align*}
for $x\in[0,1]$ and hence the third term on the right hand side of $(\ref{ineqn_Spnintz21})$ has a different sign from the left hand side. Therefore,
\begin{align}\label{ineqn_Spnmainineqnz21}
&n|z_2(b_{2j}^{\pm})-z_2(a_{2j}^{\pm})|\\
&\leq\int_{a_{2j}^{\pm}}^{b_{2j}^{\pm}}\big[ \frac{1}{2}(x-x^3)(|y_{11}'z_2'|+|z_1'y_{22}'|)+C_5x(\tau_1+C_4 \varepsilon)\,\big(\sum_{i=1}^4|z_i(x)|\big)\big]\,dx\notag\\
&\leq \frac{\sqrt{3}}{9}(\tau_1+C_2\varepsilon)\,\big(|z_2(b_{2j}^{\pm})-z_2(a_{2j}^{\pm})|+V_{a_{2j}^{\pm}}^{b_{2j}^{\pm}}(z_1)\big)+C_5(\tau_1+C_4 \varepsilon)\,(b_{2j}^{\pm}-a_{2j}^{\pm})\sum_{i=1}^4\sup_{x\in[0,1]}|z_i|,\notag
\end{align}
for any $j$, where $C_2>0$ is defined in $(\ref{ineqn_Spnrightboundy11})$,  $C_2=C_2(\delta_0)$ and $C_4=C_4(\delta_0)>0$ depend on $\tau_1$ with $\delta_0$ chosen below $(\ref{ineqn_Spnrightboundy11})$, and $C_5>0$ is independent of $\tau_1$ and $\varepsilon$. Summarizing this inequality for all $j$, we have
\begin{align}
nV(z_2)\leq \frac{2\sqrt{3}}{9}(\tau_1+C_2\varepsilon)\,\big(V(z_2)+V(z_1)\big)+2C_5(\tau_1+C_4 \varepsilon)\,\sum_{i=1}^4V(z_i).
\end{align}
Now assume that
\begin{align}\label{ineqn_Spncondition11}
\tau_1\leq\min\{\frac{1}{8C_5},\,\frac{1}{4}\}.
\end{align}
Then we choose $\varepsilon>0$ small so that
\begin{align}\label{ineqn_Spncondition21}
\varepsilon\leq \min\{\frac{1}{4C_2}, \frac{1}{8C_4C_5}\}.
\end{align}
Therefore,
\begin{align}\label{ineqn_Spnvariationz2}
V(z_2)\leq \frac{1}{n-2}(V(z_1)+V(z_3)+V(z_4)).
\end{align}
Similarly, by the equations $(\ref{equn_SpnEinstein04})$ and $(\ref{equn_SpnEinstein05})$ there exists a constant $C_6=C_6(\delta_0)>0$ depending on $\tau_1$, and $C_7>0$ independent of $\tau_1$ and $\varepsilon$ such that when
\begin{align}\label{ineqn_Spncondition12}
\tau_1\leq\min\{\frac{1}{8C_7},\,\frac{1}{4}\},
\end{align}
and
\begin{align}\label{ineqn_Spncondition22}
\varepsilon\leq \min\{\frac{1}{4C_2}, \frac{1}{8C_6C_7}\},
\end{align}
we have
\begin{align}\label{ineqn_Spnvariationz3}
V(z_3)\leq \frac{1}{n-2}(V(z_1)+V(z_2)+V(z_4)),
\end{align}
and
\begin{align}\label{ineqn_Spnvariationz4}
V(z_4)\leq \frac{1}{n-2}(V(z_1)+V(z_2)+V(z_3)).
\end{align}

We then turn to the estimate of $V(z_1)$.

Without loss of generality, we assume that $z_1(0)\geq 0$. As the discussion in \cite{Li2}, using the expansion $(\ref{equn_expansion1})$ of the Einstein metrics, for our two distinct solutions there exists a non-vanishing coefficient in the expansion $z_1(x)=\displaystyle\sum_{k=1}z_1^{(k)}(0)x^k$ at $x=0$. Hence by $(\ref{equn_SpnEinstein02})$ we can always assume that $z_1>0$ and $z_1'>0$  on the first interval of monotonicity $(0, x_1)$ of the function $z_1$.

Let $0=x_0<x_1<...<x_{k_1}<x_{k_1+1}=1$ be the list of all the local maximum points and local minimum points of $z_1$ on $x\in[0,1]$, with $k_1$ some integers. Therefore, for any $0\leq j \leq k_1$, we have that $z_1'$ keeps the sign on $x\in(x_j, x_{j+1})$ with possibly finitely many zeroes on the interval. Multiplying $x(1-x^2)$ on both sides of $(\ref{equn_SpnEinstein01})$, we have
\begin{align}\label{eqn_Spny1regular1301}
&(x(1-x^2)y_1')'-2y_1'+\frac{1}{2n^2}x(1-x^2)[n(y_1')^2\,+\,((n-1)y_2'-y_3'-y_4')^2\\
&+\,(-y_2'+(n-1)y_3'-y_4')^2\,+\,(-y_2'-y_3'+(n-1)y_4')^2+(n-3)(y_2'+y_3'+y_4')^2]=0.\notag
\end{align}
Substitute the two solutions into $(\ref{eqn_Spny1regular1301})$ and take difference, we have
\begin{align}\label{eqn_Spnz1regular1301}
0=&\big(x(1-x^2)z_1'\big)'-2z_1'+\frac{1}{2n}x(1-x^2)(y_{11}'+y_{21}')z_1'+\frac{1}{2n}x(1-x^2)\,\times\\
&[\,\big(\,(n-1)(y_{12}'+y_{22}')-y_{13}'-y_{23}'-y_{14}'-y_{24}'\,\big)\,z_2'+\,\big(\,(n-1)(y_{13}'+y_{23}')-y_{12}'-y_{22}'-y_{14}'-y_{24}'\,\big)\,z_3'\notag\\
&\,+\big(\,(n-1)(y_{14}'+y_{24}')-y_{12}'-y_{22}'-y_{13}'-y_{23}'\,\big)\,z_4'\,].\notag
\end{align}
For each $0\leq j \leq k_1$, we do integration of $(\ref{eqn_Spnz1regular1301})$ on the interval $x\in [x_j,x_{j+1}]$,
\begin{align}\label{eqn_Spnz1regularint1301}
&2(z_1(x_{j+1})- z_1(x_j)) -\frac{1}{2n}\,\int_{x_j}^{x_{j+1}}\,x(1-x^2)(y_{11}'+y_{21}')z_1' dx\\
=&\,\frac{1}{2n}\int_{x_j}^{x_{j+1}}x(1-x^2)\,[\,\big(\,(n-1)(y_{12}'+y_{22}')-y_{13}'-y_{23}'-y_{14}'-y_{24}'\,\big)\,z_2'\notag\\
+&\big((n-1)(y_{13}'+y_{23}')-y_{12}'-y_{22}'-y_{14}'-y_{24}'\big)\,z_3'+\big((n-1)(y_{14}'+y_{24}')-y_{12}'-y_{22}'-y_{13}'-y_{23}'\big)\,z_4'\,] dx.\notag
\end{align}
By $(\ref{ineqn_Spnleftboundy11})-(\ref{ineqn_Spnleftboundy012})$ we have
\begin{align*}
0<x(1-x^2)(y_{11}'+y_{21}')\leq \frac{4\sqrt{3}}{9}(\tau_1+C_2\varepsilon),
\end{align*}
and
\begin{align*}
x(1-x^2)\,|(n-1)(y_{1i}'+y_{2i}')-y_{1j}'-y_{2j}'-y_{1k}'-y_{2k}'|\leq \frac{4(n+1)\sqrt{3}}{9}(\tau_1+C_2\varepsilon)
\end{align*}
for any $\{i,j,k\}=\{2,3,4\}$ and $x\in[0,1]$. Substituting these estimates into $(\ref{eqn_Spnz1regularint1301})$, we have
\begin{align*}
\big(2-\frac{2\sqrt{3}}{9n}(\tau_1+C_2\varepsilon)\big)|z_1(x_{j+1})- z_1(x_j)|
\leq&\,\frac{2(n+1)\sqrt{3}}{9n}(\tau_1+C_2\varepsilon)\int_{x_j}^{x_{j+1}}\,[\,\,|z_2'|+|z_3'|+|z_4'|\,] dx\\
\leq&\,\frac{2(n+1)\sqrt{3}}{9n}(\tau_1+C_2\varepsilon) \sum_{i=2}^4V_{x_j}^{x_{j+1}}(z_i).\notag
\end{align*}
Summarizing this inequality for all $j$, one has
\begin{align*}
\big(1-\frac{\sqrt{3}}{9n}(\tau_1+C_2\varepsilon)\big)V(z_1)
\leq\,\frac{2(n+1)\sqrt{3}}{9n}(\tau_1+C_2\varepsilon) \sum_{i=2}^4V(z_i),
\end{align*}
with $C_2=C_2(\delta_0)>0$  in $(\ref{ineqn_Spnrightboundy11})$ depending on $\tau_1$. Assume that
\begin{align}\label{ineqn_Spncondition13}
\tau_1\leq \frac{1}{2(n+1)},
\end{align}
and then we choose $\varepsilon>0$ so that
\begin{align}\label{ineqn_Spncondition23}
\varepsilon\leq \frac{1}{2(n+1)C_2},
\end{align}
and hence,
\begin{align}\label{ineqn_Spnvariationz1}
V(z_1) \leq\,\frac{1}{n}\sum_{i=2}^4V(z_i).
\end{align}

 By Theorem \ref{EHBoundary}, for $\varepsilon>0$ small, there exists a small constant $\tau_0=\tau_0(\varepsilon)\in(0, 1)$ such that when
\begin{align}
|1-t_1(0)|+|1-t_2(0)|+|1-t_3(0)|\leq \tau_0,
\end{align}
we have that
\begin{align*}
\sup_M|W|_g\leq \varepsilon.
\end{align*}

Now let $\tau=\min\{\tau_1,\tau_0\}$ with $\tau_1$ and $\varepsilon$ satisfying $(\ref{ineqn_Spncondition11}),\,(\ref{ineqn_Spncondition21}),\,(\ref{ineqn_Spncondition12}),\,(\ref{ineqn_Spncondition22}),\,(\ref{ineqn_Spncondition13})$ and $(\ref{ineqn_Spncondition23})$. Thus if
\begin{align*}
|1-t_1(0)|+|1-t_2(0)|+|1-t_3(0)|\leq \tau,
\end{align*}
then $(\ref{ineqn_Spnvariationz2}),\,(\ref{ineqn_Spnvariationz3}),\,(\ref{ineqn_Spnvariationz4})$ and $(\ref{ineqn_Spnvariationz1})$ hold, and hence,
\begin{align*}
z_i(x)=0
\end{align*}
for $x\in[0,1]$ and $1\leq i \leq 4$. Therefore these two solutions coincide. This completes the proof of the theorem.

\end{proof}

\begin{appendix}
\section{}\label{Appendix1}
Let $\mathbb{S}^{15}=\text{Spin}(9)/\text{Spin}(7)$ and $\hat{g}$ be a Spin$(9)$-invariant metric on $\mathbb{S}^{15}$. Here $\text{Spin}(7)$ and $\text{Spin}(9)$ are the Spin groups. In this appendix, when the conformal infinity is $(\mathbb{S}^{15},[\hat{g}])$, we first deform the prescribed conformal infinity problem of the negatively curved
CCE metrics into a two-point boundary value problem of a system of ODEs as we did for the case of $\text{SU}(k+1)$-invariant metrics on $\mathbb{S}^{2k+1}$ in \cite{Li2}. Then we generalize the uniqueness and existence results to this case.

Let $(M^{16},g)$ be a CCE manifold which is Hadamard with its conformal infinity $(\mathbb{S}^{15},[\hat{g}])$.  Let $p_0\in M$ be the center of gravity, $r$ be the distance function to $p_0$ on $(M,g)$ and $x = e^{-r}$ the geodesic defining function about $C\hat{g}$ with some constant
$C > 0$ as discussed \cite{Li} and \cite{Li2}.  Recall that the CCE metric has the forms $(\ref{eqn_metrictwocomponents})$ and $(\ref{equn_splittingmetric01})$ along a fixed geodesic $\gamma$ from $p_0$ to infinity, with $\bar{h}$ diagonal under the polar coordinate $(r,\theta)$ centered at $p_0$ in a neighborhood of $\gamma$:
\begin{align*}
\bar{h}=\text{diag}(I_1,..,I_{15}).
\end{align*}

Considering $\mathbb{S}^{15}$ as the unit ball in $\mathbb{R}^{16}$, we use the coordinate $(x,y)=(x_1,..,x_8,y_1,..,y_8)$. Let $\mathfrak{p}$ be the Lie algebra of $\text{Spin}(9)$, with the splitting $\mathfrak{p}=\mathfrak{p}_1\oplus\mathfrak{p}_2$ as in \cite{Friedrich}, \cite{OPPV} and \cite{VZ}, from which one can derive a basis of $\mathfrak{p}$ as follows:\\
The basis of $\mathfrak{p}_2$ can be chosen as
\begin{align*}
&X_1=\frac{1}{2}(y_8,y_7,-y_6,-y_5,y_4,y_3,-y_2,-y_1,x_8,x_7,-x_6,-x_5,x_4,x_3,-x_2,-x_1),\\
&X_2=\frac{1}{2}(-y_7,y_8,y_5,-y_6,-y_3,y_4,y_1,-y_2,-x_7,x_8,x_5,-x_6,-x_3,x_4,x_1,-x_2),\\
&X_3=\frac{1}{2}(-y_6,y_5,-y_8,y_7,-y_2,y_1,-y_4,y_3,-x_6,x_5,-x_8,x_7,-x_2,x_1,-x_4,x_3),\\
&X_4=\frac{1}{2}(-y_5,-y_6,-y_7,-y_8,y_1,y_2,y_3,y_4,-x_5,-x_6,-x_7,-x_8,x_1,x_2,x_3,x_4),\\
&X_5=\frac{1}{2}(-y_3,-y_4,y_1,y_2,y_7,y_8,-y_5,-y_6,-x_3,-x_4,x_1,x_2,x_7,x_8,-x_5,-x_6),\\
&X_6=\frac{1}{2}(y_4,-y_3,y_2,-y_1,-y_8,y_7,-y_6,y_5,x_4,-x_3,x_2,-x_1,-x_8,x_7,-x_6,x_5),\\
&X_7=\frac{1}{2}(y_2,-y_1,-y_4,y_3,-y_6,y_5,y_8,-y_7,x_2,-x_1,-x_4,x_3,-x_6,x_5,x_8,-x_7),\\
&X_8=\frac{1}{2}(-y_1,-y_2,-y_3,-y_4,-y_5,-y_6,-y_7,-y_8,x_1,x_2,x_3,x_4,x_5,x_6,x_7,x_8),
\end{align*}
while the basis of $\mathfrak{p}_1$ is chosen as
\begin{align*}
&Y_1=-[X_1,X_5]-[X_2,X_6]-[X_3,X_7]-[X_4,X_8],\\
&Y_2=-[X_1,X_7]-[X_2,X_8]+[X_3,X_5]+[X_4,X_6],\\
&Y_3=-[X_1,X_3]+[X_2,X_4]+[X_5,X_7]-[X_6,X_8],\\
&Y_4=-[X_1,X_6]+[X_2,X_5]+[X_3,X_8]-[X_4,X_7],\\
&Y_5=-[X_1,X_8]+[X_2,X_7]-[X_3,X_6]+[X_4,X_5],\\
&Y_6=-[X_1,X_2]-[X_3,X_4]+[X_5,X_6]+[X_7,X_8],\\
&Y_7=-[X_1,X_4]-[X_2,X_3]+[X_5,X_8]+[X_6,X_7].
\end{align*}
It holds that
\begin{align*}
[[X_i,X_j],X_k]=\delta_{ik}X_j-\delta_{jk}X_i,
\end{align*}
for $1\leq i,j,k\leq 8$.

\begin{Remark}
$\text{Spin}(9)$ is the double cover (also the universal cover) of SO$(9)$. They share the same Lie algebra locally. The basis listed in Page $4$ in \cite{VZ} can be used to check whether the basis one chooses is correct. 
 One has the information that the basis of $\text{Spin}(9)$ in $\mathbb{R}^{16}$, $J_1,..,J_9$ ( denoted as $I_1,..,I_9$ in \cite{Friedrich}), acts linearly on a $9$-dimensional subspace $V(q)$ of $T_q\mathbb{R}^{16}$ invariantly. Here $V(q)$ contains the vector $q$ and is not a subset of $T_q\mathbb{S}^{15}$ (It contains an $8$-dimensional sub-space $U(q)$ of $T_q\mathbb{S}^{15}$). $V$ is a sub-bundle on $\mathbb{S}^{15}$ and $\text{Spin}(9)$ acts on it invariantly, see Page $4$, Page $6$ and Page $7$ in \cite{OPPV}, and see also \cite{Friedrich}. In Page $21$ in \cite{Friedrich}, one has that $\text{spin}(9)=Lin(J_{\alpha}J_{\beta}:\,\alpha<\beta)$. $\text{Spin}(9)$ is of dimension $C_9^2=36$. 
 The group operation in $\text{Spin}(9)$ can be considered as the products of the $16\times 16$ matrices $J_1,.., J_9$ (see Page $23$ in \cite{Friedrich}). Also, the basis $X_1,..,X_8$ are given by $J_1J_9,..,J_8J_9$ as in \cite{VZ}. Here $J_1J_9$ is the product of the matrices $J_1$ and $J_9$. To write them as vector fields in $\mathbb{R}^{16}$, just let $X_1^T=J_1J_9(x,y)^T$ with the matrices $J_1,J_9$ and the vector $(x,y)\in\mathbb{R}^{16}$. In Page $4$ of \cite{VZ}, the $9\times9$ matrix $E_{ij}=(\delta_{ij})-(\delta_{ji})$. ( The product $J_{\alpha}J_{\beta}$ of the matrices $J_{\alpha}$ and $J_{\beta}$ in \cite{Friedrich} is corresponding to $E_{\alpha,\beta}$ in \cite{VZ}.) It is easy to check that $[E_{i,9},E_{j,9}]=-E_{ij}$ and $[[E_{i,9},E_{j,9}],E_{k,9}]=\delta_{i,k}E_{j,9}-\delta_{jk}E_{i,9}$. Then the basis of $\mathfrak{p}_1$ satisfies the propositions relating to the basis in $\mathfrak{p}_2$ as above (the Lie bracket structures), see Page $4$ in \cite{VZ}. There is no preference of choice of the order of the indices in the coordinates $x_i$.
\end{Remark}
Therefore,
\begin{align*}
&Y_1=(-x_5,-x_6,-x_8,-x_7,x_1,x_2,x_4,x_3,0,0,-y_8+y_7,y_8-y_7,0,0,y_4-y_3,y_3-y_4),\\
&Y_2=(-x_7,x_8,-x_6,x_5,-x_4,x_3,x_1,-x_2,0,0-y_5-y_6,y_5+y_6,y_3-y_4,y_3-y_4,0,0),\\
&Y_3=(x_3,-x_4,-x_1,x_2,-x_8,x_7,-x_6,x_5,y_3-y_4,y_3-y_4,-y_1-y_2,y_1+y_2,0,0,0,0),\\
&Y_4=(x_6,-x_5,-x_7,x_8,x_2,-x_1,x_3,-x_4,0,0,-y_7-y_8,y_7+y_8,0,0,y_3-y_4,y_3-y_4),\\
&Y_5=(x_8,x_7,-x_5,-x_6,x_3,x_4,-x_2,-x_1,0,0,y_6-y_5,y_5-y_6,y_3-y_4,y_4-y_3,0,0),\\
&Y_6=(-x_2,x_1,-x_4,x_3,x_6,-x_5,-x_8,x_7,0,0,-2y_4,2y_3,0,0,0,0),\\
&Y_7=(x_4,x_3,-x_2,-x_1,-x_7,-x_8,x_5,x_6,y_4-y_3,y_3-y_4,y_1-y_2,y_2-y_1,0,0,0,0).
\end{align*}
Notice that $X_1,..,X_8$ and $Y_1,..,Y_7$ form a basis of $T_q\mathbb{S}^{15}$ at $q=(1,0,..,0)$. In a neighborhood of $q$ on $\mathbb{S}^{15}$ we use the local coordinate $\theta=(\theta^1,..,\theta^{15})=(x_2,..,x_8,y_1,..,y_8)$. Then $q$ is denoted as $\theta_0=(0,..,0)$ under this local coordinate. In order to be more convenient for the calculation, we define the $\text{Spin}(9)$-invariant vector fields
\begin{align*}
&Z_1=Y_6,\,Z_2=Y_3,\,Z_3=Y_7,\,Z_4=Y_1,\,Z_5=Y_4,\,Z_6=Y_2,\,Z_7=Y_5,\\
&Z_8=X_8,\,Z_9=X_7,\,Z_{10}=X_5,\,Z_{11}=X_6,\,Z_{12}=X_4,\,Z_{13}=X_3,\,Z_{14}=X_2,\,Z_{15}=X_1.
\end{align*}
Under the local coordinate $(\theta^1,..,\theta^{15})$,
\begin{align*}
&Z_1=(x_1,-\theta_3,\theta_2,\theta_5,-\theta_4,-\theta_7,\theta_6,0,0,-2\theta_{11},2\theta_{10},0,0,0,0),\\
&Z_2=(-\theta_3,-x_1,\theta_1,-\theta_7,\theta_6,-\theta_5,\theta_4,\theta_{10}-\theta_{11},\theta_{10}-\theta_{11},-\theta_8-\theta_9,\theta_8+\theta_9,0,0,0,0),\\
&Z_3=(\theta_2,-\theta_1,-x_1,-\theta_6,-\theta_7,\theta_4,\theta_5,\theta_{11}-\theta_{10},\theta_{10}-\theta_{11},\theta_8-\theta_9,\theta_9-\theta_8,0,0,0,0),\\
&Z_4=(-\theta_5,-\theta_7,-\theta_6,x_1,\theta_1,\theta_3,\theta_2,0,0,\theta_{14}-\theta_{15},\theta_{15}-\theta_{14},0,0,\theta_{11}-\theta_{10},\theta_{10}-\theta_{11}),\\
&Z_5=(-\theta_4,-\theta_6,\theta_7,\theta_1,-x_1,\theta_2,-\theta_3,0,0,-\theta_{14}-\theta_{15},\theta_{14}+\theta_{15},0,0,\theta_{10}-\theta_{11},\theta_{10}-\theta_{11}),\\
&Z_6=(\theta_7,-\theta_5,\theta_4,-\theta_3,\theta_2,x_1,-\theta_1,0,0,-\theta_{12}-\theta_{13},\theta_{12}+\theta_{13},\theta_{10}-\theta_{11},\theta_{10}-\theta_{11},0,0),\\
&Z_7=(\theta_6,-\theta_4,-\theta_5,\theta_2,\theta_3,-\theta_1,-x_1,0,0,\theta_{13}-\theta_{12},\theta_{12}-\theta_{13},\theta_{10}-\theta_{11},\theta_{11}-\theta_{10},0,0),\\
&Z_8=\frac{1}{2}(-\theta_9,-\theta_{10},-\theta_{11},-\theta_{12},-\theta_{13},-\theta_{14},-\theta_{15},x_1,\theta_1,\theta_2,\theta_3,\theta_4,\theta_5,\theta_6,\theta_7),\\
&Z_9=\frac{1}{2}(-\theta_8,-\theta_{11},\theta_{10},-\theta_{13},\theta_{12},\theta_{15},-\theta_{14},\theta_1,-x_1,-\theta_3,\theta_2,-\theta_5,\theta_4,\theta_7,-\theta_6),\\
\end{align*}
\begin{align*}
&Z_{10}=\frac{1}{2}(-\theta_{11},\theta_8,\theta_9,\theta_{14},\theta_{15},-\theta_{12},-\theta_{13},-\theta_2,-\theta_3,x_1,\theta_1,\theta_6,\theta_7,-\theta_4,-\theta_5),\\
&Z_{11}=\frac{1}{2}(-\theta_{10},\theta_9,-\theta_8,-\theta_{15},\theta_{14},-\theta_{13},\theta_{12},\theta_3,-\theta_2,\theta_1,-x_1,-\theta_7,\theta_6,-\theta_5,\theta_4),\\
&Z_{12}=\frac{1}{2}(-\theta_{13},-\theta_{14},-\theta_{15},\theta_8,\theta_9,\theta_{10},\theta_{11},-\theta_4,-\theta_5,-\theta_6,-\theta_7,x_1,\theta_1,\theta_2,\theta_3),\\
&Z_{13}=\frac{1}{2}(\theta_{12},-\theta_{15},\theta_{14},-\theta_9,\theta_8,-\theta_{11},\theta_{10},-\theta_5,\theta_4,-\theta_7,\theta_6,-\theta_1,x_1,-\theta_3,\theta_2),\\
&Z_{14}=\frac{1}{2}(\theta_{15},\theta_{12},-\theta_{13},-\theta_{10},\theta_{11},\theta_8,-\theta_9,-\theta_6,\theta_7,\theta_4,-\theta_5,-\theta_2,\theta_3,x_1,-\theta_1),\\
&Z_{15}=\frac{1}{2}(\theta_{14},-\theta_{13},-\theta_{12},\theta_{11},\theta_{10},-\theta_9,-\theta_8,\theta_7,\theta_6,-\theta_5,-\theta_4,\theta_3,\theta_2,-\theta_1,-x_1).
\end{align*}

At $\theta=\theta_0$, we have the matrix
\begin{align*}
\left(\begin{matrix}&Z_1(\theta_0)\\ &. \\ &. \\ &Z_{15}(\theta_0) \end{matrix}\right)=\text{diag}(1,-1,-1,1,-1,1,-1,\frac{1}{2},-\frac{1}{2},\frac{1}{2},-\frac{1}{2},\frac{1}{2},\frac{1}{2},\frac{1}{2},-\frac{1}{2}).
\end{align*}

In a neighborhood of $q$ under the coordinate $(\theta^1,..,\theta^{15})$, define the inverse of the matrix with elements $Z_i^j$ for $1\leq i, j \leq n$ as
\begin{align*}
\left(\begin{matrix}&&&\\ &Q_i^j&&\\ &&& \end{matrix}\right)=\left(\begin{matrix}&&&\\ &Z_i^j&&\\ &&& \end{matrix}\right)^{-1}.
\end{align*}
We denote
\begin{align*}
C_{ij}^p=Q_i^q\frac{\partial}{\partial \theta^j}Z_q^p,
\end{align*}
and
\begin{align}\label{equn_twotensor}
T_{ij}^p = - T_{ji}^p=C_{ij}^p- C_{ji}^p=Q_i^qQ_j^m[Z_m, Z_q]^p,
\end{align}
with $[Z_m, Z_q]$ the Lie bracket of $Z_m$ and $Z_q$.
Notice that $C_{ij}^p$ and $T_{ij}^p$ are independent of $r$ and the metric. Then as in \cite{Li}, we have
\begin{align}
&\label{equn_tangentspace}\frac{\partial}{\partial \theta^q}g_{ij}=-C_{qi}^mg_{mj}-C_{qj}^mg_{mi},\\
&\Gamma_{ij}^p(g_r)=\frac{1}{2}[-(C_{ij}^p+C_{ji}^p)+g^{pq}(-C_{iq}^mg_{mj}-C_{jq}^mg_{mi}+C_{qi}^mg_{mj}+C_{qj}^mg_{mi})].
\end{align}
The Ricci curvature of $g_r$ on the geodesic spheres has the expression (see \cite{Li}) 
\begin{align}
R_{ij}(g_r)&\label{equn_Riccitensor}=\frac{1}{2}(\frac{\partial}{\partial \theta^p}T_{ij}^p+C_{ip}^qT_{qj}^p+C_{qj}^pT_{ip}^q+C_{qp}^pT_{ji}^q)-\frac{1}{2}g^{pq}(\frac{\partial}{\partial \theta^p}T_{iq}^m+C_{ip}^sT_{sq}^m+C_{pq}^sT_{is}^m-C_{ps}^mT_{iq}^s)g_{mj}\\
&-\frac{1}{2}g^{pq}(\frac{\partial}{\partial \theta^p}T_{jq}^m+C_{jp}^sT_{sq}^m+C_{pq}^sT_{js}^m-C_{ps}^mT_{jq}^s)g_{mi}+\frac{1}{4}T_{pi}^sT_{sj}^p-\frac{1}{4}g^{pq}T_{pi}^sT_{sq}^mg_{mj}\notag\\
&-\frac{1}{4}g^{pq}T_{pj}^sT_{sq}^mg_{mi}-\frac{1}{2}g^{sq}T_{ps}^pT_{iq}^mg_{mj}-\frac{1}{2}g^{sq}T_{ps}^pT_{jq}^mg_{mi}+\frac{1}{4}g^{pq}T_{pj}^sT_{iq}^mg_{sm}+\frac{1}{4}g^{pq}T_{pi}^sT_{jq}^mg_{sm}\notag\\
&-\frac{1}{4}g^{pl}(T_{jl}^mg_{ms}+T_{sl}^mg_{mj})g^{sq}(T_{pq}^mg_{mi}+ T_{iq}^mg_{mp})\notag.
\end{align}

Notice that
\begin{align*}
g_r=\sinh^2(r)\bar{h}=\frac{x^{-2}(1-x^2)^2}{4}\bar{h}_{ij}d\theta^id\theta^j.
\end{align*}

By direct calculation, at $\theta=\theta_0$, we have
\begin{align*}
&C_{12}^3=C_{15}^4=C_{16}^7=1,\,\,C_{1,11}^{10}=-2,\\
&C_{23}^1=C_{27}^4=C_{25}^6=C_{2,11}^8=C_{2,11}^9=C_{28}^{10}=C_{29}^{10}=1,\\
&C_{31}^2=C_{36}^4=C_{37}^5=C_{3,10}^8=C_{3,11}^9=C_{3,9}^{10}=C_{3,8}^{11}=1,\\
&C_{41}^5=C_{43}^6=C_{42}^7=C_{4,14}^{10}=C_{4,15}^{11}=C_{4,11}^{14}=C_{4,10}^{15}=1,\\
&C_{54}^1=C_{56}^2=C_{53}^7=C_{5,11}^{14}=C_{5,11}^{15}=C_{5,14}^{10}=C_{5,15}^{10}=1,\\
&C_{62}^5=C_{64}^3=C_{67}^1=C_{6,10}^{12}=C_{6,10}^{13}=C_{6,12}^{11}=C_{6,13}^{11}=1,\\
&C_{71}^6=C_{74}^2=C_{75}^3=C_{7,10}^{13}=C_{7,11}^{12}=C_{7,12}^{10}=C_{7,13}^{11}=1,\\
&C_{81}^9=C_{82}^{10}=C_{83}^{11}=C_{84}^{12}=C_{85}^{13}=C_{86}^{14}=C_{87}^{15}=1,\\
&C_{98}^1=C_{9,11}^2=C_{9,13}^4=C_{9,14}^7=C_{93}^{10}=C_{95}^{12}=C_{96}^{15}=1,\\
&C_{10,8}^2=C_{10,9}^3=C_{10,14}^4=C_{10,15}^5=C_{10,1}^{11}=C_{10,6}^{12}=C_{10,7}^{13}=1,\\
&C_{11,10}^1=C_{11,8}^3=C_{11,15}^4=C_{11,13}^6=C_{11,2}^9=C_{11,7}^{12}=C_{11,5}^{14}=1,\\
&C_{12,8}^4=C_{12,9}^5=C_{12,10}^6=C_{12,11}^7=C_{12,1}^{13}=C_{12,2}^{14}=C_{12,3}^{15}=1,\\
&C_{13,12}^1=C_{13,14}^3=C_{13,8}^5=C_{13,10}^7=C_{13,4}^9=C_{13,6}^{11}=C_{13,2}^{15}=1,\\
&C_{14,15}^1=C_{14,12}^2=C_{14,11}^5=C_{14,8}^6=C_{14,7}^9=C_{14,4}^{10}=C_{14,3}^{13}=1,\\
&C_{15,13}^2=C_{15,12}^3=C_{15,9}^6=C_{15,8}^7=C_{15,5}^{10}=C_{15,4}^{11}=C_{15,1}^{14}=1,
\end{align*}
with $C_{ij}^k=-C_{ik}^j$, and $C_{ij}^k=0$ for the remaining cases.

Direct calculation yields
\begin{align*}
&\frac{\partial}{\partial \theta^p}T_{ii}^p=0,\\
&\frac{\partial}{\partial \theta^p}T_{ip}^i=-C_{ip}^qQ_q^qQ_p^p[Z_p,Z_q]^i+Q_i^iQ_p^p\frac{\partial}{\partial \theta^p}([Z_p,Z_i]^i).
\end{align*}
The terms $[Z_p,Z_q]^i, \frac{\partial}{\partial \theta^p}([Z_p,Z_i]^i)$ can be calculated directly using the explicit formulas of $X_t,Y_s$ or $Z_m$ at $\theta=\theta_0$. We only need the calculation of the case $i=1$ and $i=15$: $R_{11}(g_r)$ and $R_{15,15}(g_r)$.
\begin{align*}
&\frac{\partial}{\partial \theta^p}T_{15,p}^{15}=0,\,\,1\leq p\leq 7,\\
&\frac{\partial}{\partial \theta^p}T_{15,p}^{15}=-1,\,\,\text{for}\,\,p=8,9,12,13,14,15,\\
&\frac{\partial}{\partial \theta^p}T_{15,p}^{15}=-2,\,\,\text{for}\,\,p=10,11,
\end{align*}
at $\theta=\theta_0$ and
\begin{align*}
&\frac{\partial}{\partial \theta^p}T_{1p}^1=0,\,\,2\leq p\leq 7,\\
&\frac{\partial}{\partial \theta^p}T_{1p}^1=-1,\,\,\text{for}\,\,p=8,9,12,13,14,15,\\
&\frac{\partial}{\partial \theta^p}T_{1p}^1=-3,\,\,\text{for}\,\,p=10,11.
\end{align*}
Therefore, let $t=\frac{I_8}{I_1}$, and by the same argument in \cite{Li}, we have
\begin{align*}
g_r=I_1 \text{diag}(1,1,1,1,1,1,1,t,t,t,t,t,t,t,t),
 \end{align*}
 and
  \begin{align*}
  \text{Ric}(g_r)=\text{diag}(6+8t^{-2},..,6+8t^{-2},28-14t^{-1},..,28-14t^{-1}).
   \end{align*}
   That is, $R_{ii}(g_r)=6+8t^{-2}$ for $1\leq i \leq 7$, and $R_{ii}(g_r)=28-14t^{-1}$ for $8\leq i \leq 15$.
   \begin{Remark}
   There is a different choice of the vector fields $\{Y_1,..,Y_7\}$ in $(3.1)$ in \cite{Parton-Piccinni}, and one can check directly that the same formula of $\text{Ric}(g_r)$ is derived.
   \end{Remark}

   Let $K=I_1^7I_8^8$. Let $y_1=\log(K)$ and $y_2=\log(t)$, and hence we have
\begin{align}\label{equn_changevariablesSUn}
I_1=(Kt^{-8})^{\frac{1}{15}},\,\,I_8=(Kt^7)^{\frac{1}{15}}.
\end{align}
Therefore, by the same calculation as in \cite{Li} and \cite{Li2}, the boundary value problem of the Einstein metrics
becomes
\begin{align}
&\label{equn_SpinEinstein01}y_1''+\frac{1}{30}[(y_1')^2+ 56(y_2')^2]-x^{-1}(1+3x^2)(1-x^2)^{-1}y_1'=0,\\
&\label{equn_SpinEinstein02}y_1''-[29+31x^2\,]\,x^{-1}(1-x^2)^{-1}y_1'+\frac{1}{2}(y_1')^2\\
&+8\times14(1-x^2)^{-2}[15-K^{-\frac{1}{15}}t^{-\frac{22}{15}}(3t^2+16t-4)]=0,\notag\\
&\label{equn_SpinEinstein03}y_2''-[14+ 16x^2\,]\,x^{-1}(1-x^2)^{-1}y_2'+\frac{1}{2}y_1'y_2'+ 16(1-x^2)^{-2}K^{-\frac{1}{15}}t^{-\frac{22}{15}}(3t^2-14t+11)=0,
\end{align}
for $y_1(x),y_2(x)\in C^{\infty}([0,1])$ with the boundary condition
\begin{align}\label{equn_SpinBV01}
t(0)=\frac{\lambda_2}{\lambda_1},\,\,K(1)=\phi(1)=1,\,\,y_1'(0)=y_2'(0)=y_1'(1)=y_2'(1)=0.
\end{align}
Combining $(\ref{equn_SpinEinstein01})$ and $(\ref{equn_SpinEinstein02})$, we have
\begin{align}\label{equn_SpinEinstein04}
(y_1')^2-4(y_2')^2-60x^{-1}(1+x^2)(1-x^2)^{-1}y_1'+16\times 15(1-x^2)^{-2}(15-K^{-\frac{1}{15}}t^{-\frac{22}{15}}(3t^2+16t-4))=0.
\end{align}
By $(\ref{equn_expansion1})$, we have the expansion of $y_1$ and $y_2$ at $x=0$, which can also be done directly using the system $(\ref{equn_SpinEinstein01})-(\ref{equn_SpinEinstein03})$ and the boundary data $(\ref{equn_SpinBV01})$. Let $\Phi(x)$ be the function on the left hand side of the equation $(\ref{equn_SpinEinstein04})$. Take derivative of $\Phi$ and use the equations $(\ref{equn_SpinEinstein02})$ and $(\ref{equn_SpinEinstein03})$ we have
\begin{align}\label{equn_Spin2-21}
\Phi'+(y_1'-2x^{-1}(14+16x^2)(1-x^2)^{-1})\Phi=0.
\end{align}
Consider $y_1'$ as a given function. Using the expansion $(\ref{equn_expansion1})$, we can derive that $(\ref{equn_Spin2-21})$ has a unique solution $\Phi=0$, which is $(\ref{equn_SpinEinstein04})$. Therefore, $(\ref{equn_SpinEinstein02})$ and $(\ref{equn_SpinEinstein03})$ combining with the expansion of the Einstein metric imply $(\ref{equn_SpinEinstein04})$. Similarly, any two of the equations $(\ref{equn_SpinEinstein01})-(\ref{equn_SpinEinstein03})$ and $(\ref{equn_SpinEinstein04})$ combining with the boundary expansion of the Einstein metric give the other two equations. Notice that the coefficients of the expansion of the metric can be solved inductively by the equations $(\ref{equn_SpinEinstein02})-(\ref{equn_SpinEinstein03})$ and the initial data $(\ref{equn_SpinBV01})$ before the order $x^{15}$.

\begin{lem}\label{lem_monotonicitySpin101}
For the initial data $t(0)\neq 1$, assume $(y_1,y_2)$ is a global solution of the boundary value problem $(\ref{equn_SpinEinstein01})-(\ref{equn_SpinBV01})$. Then we have $y_1'(x)>0$ for $x\in(0,1)$. Also, we have the inequality
\begin{align}\label{ineqn_Spiny2upperbd}
t(x)\geq \min \{1,t(0)\}
\end{align}
for $x\in[0,1]$. Moreover, if $y_2'$ has a zero on $x\in(0,1)$, assume $x_2$ is the largest zero of $y_2'$ on $x\in(0,1)$, then for $x\in(x_2,1)$ we have
\begin{align}\label{ineqn_Spiny2monotonicity1-1}
y_2'(x)<0.
\end{align}
\end{lem}

\begin{proof}
The proof of the inequality $y_1'(x)>0$ on $x\in(0,1)$ is the same as Lemma \ref{lem_monotonicitySpn01}.

The inequality $(\ref{ineqn_Spiny2upperbd})$ holds if $t$ is monotone on $x\in(0,1)$. Now we assume that $\bar{x}$ is a local minimum point of $t$ on $x\in(0,1)$, and hence $y_2'(\bar{x})=0$ and $y_2''(\bar{x})\leq 0$. By $(\ref{equn_SpinEinstein03})$,
\begin{align*}
y_2''(\bar{x})=-16(1-x^2)^{-2}K^{-\frac{1}{15}}t^{-\frac{22}{15}}(t-1)(3t-11)\big|_{x=\bar{x}}\,\,\geq 0.
\end{align*}
Therefore,
\begin{align}\label{ineqn_Spiny2upperbd1}
1 \leq t(\bar{x})\leq \frac{11}{3}.
\end{align}
 This proves the inequality $(\ref{ineqn_Spiny2upperbd})$. Similarly, at a local maximum point $\tilde{x}$ of $t$ on $x\in(0,1)$, we have
\begin{align*}
t(\tilde{x})\geq \frac{11}{3},\,\,\,\text{or}\,\,\,\,t(\tilde{x})\leq 1.
\end{align*}
Since $t(1)=1$ and $t$ has isolated zeroes, by $(\ref{ineqn_Spiny2upperbd1})$ we have that
\begin{align}\label{ineqn_Spiny2lowerbd}
t(\tilde{x})\geq \frac{11}{3}.
\end{align}
Let $x_2$ be the largest zero of $y_2'$ on $x\in(0,1)$. Then $(1-t)y_2'>0$ on $x\in(x_2,1)$. We multiply $K^{\frac{1}{2}}x^{-14}(1-x^2)^{15}$ on both sides of $(\ref{equn_SpinEinstein03})$ to have
\begin{align*}
(K^{\frac{1}{2}}x^{-14}(1-x^2)^{15}y_2')'
&+ 16x^{-14}(1-x^2)^{13}K^{\frac{1}{2}-\frac{1}{15}}t^{-\frac{22}{15}}(t-1)(3t-11)=0,
\end{align*}
 and do integration on $x\in[x_2,1]$ to have
\begin{align}\label{ineqn_Spiny2lowerbd1}
16 \int_{x_2}^1x^{-14}(1-x^2)^{13}K^{\frac{1}{2}-\frac{1}{15}}t^{-\frac{22}{15}}(t-1)(3t-11)dx=0.
\end{align}
If $t-1<0$ on $x\in(x_2,1)$, then $(t-1)(3t-11)>0$, contradicting with $(\ref{ineqn_Spiny2lowerbd1})$. Therefore, $t-1>0$ and $y_2'<0$ on $x\in(x_2,1)$, and moreover,
\begin{align*}
t(x_2)>\frac{11}{3}.
\end{align*}
\end{proof}

By $(\ref{equn_SpinEinstein04})$ and $(\ref{ineqn_Spiny2upperbd})$, for $t(0)\geq \frac{-8+2\sqrt{19}}{3}$ ( to guarantee that $(3t^2+16t-4)>0$ 
), we have
\begin{align}
y_1'&\label{equn_Spiny1lowerorder1}=2x^{-1}(1-x^2)^{-1}[15(1+x^2)-\sqrt{15^2(1+x^2)^2+x^2(1-x^2)^2(y_2')^2-60x^2\Upsilon(x)}\,\,]\\
&=2x^{-1}(1-x^2)^{-1}[15(1+x^2)\notag\\
&-\sqrt{15^2(1-x^2)^2+x^2(1-x^2)^2(y_2')^2+60x^2 K^{-\frac{1}{15}}t^{-\frac{22}{15}}(3t^2+16t-4)}\,\,],\notag
\end{align}
for $x\in(0,1)$, where $\Upsilon(x)=15-K^{-\frac{1}{15}}t^{-\frac{22}{15}}(3t^2+16t-4)$. Since $y_1'>0$ for $x\in(0,1)$, it is clear by $(\ref{equn_Spiny1lowerorder1})$ that
\begin{align}\label{inequn_boundaryYamabeconstantSpin1}
\Upsilon(x)=15-K^{-\frac{1}{15}}t^{-\frac{22}{15}}(3t^2+16t-4)>x^2(1-x^2)^2(y_2')^2\,\geq\, 0,
\end{align}
for $ x\in(0,1)$.
By Lemma \ref{lem_monotonicitySpin101} and $(\ref{equn_Spiny1lowerorder1})$, we have that
\begin{align}\label{equn_Spiny1lowerorder21}
0 < y_1'\leq 30 x^{-1}(1-x^2)^{-1}[(1+x^2)
-\sqrt{(1-x^2)^2}\,\,]\leq 60 x(1-x^2)^{-1},
\end{align}
for $x\in(0,1)$, provided $t(0)\neq 1$. Also by $(\ref{inequn_boundaryYamabeconstantSpin1})$, 
\begin{align*}
15\geq 15 K^{\frac{1}{15}}\geq t^{-\frac{22}{15}}(3t^2+16t-4)
\end{align*}
for $x\in[0,1]$, and hence $t$ has the upper bound
\begin{align}\label{ineqn_Spiny2lowerbd1-2}
t(x)<t_0,
\end{align}
with $t_0>0$ some constant, for $x\in[0,1]$. When $t\geq 1$, $t^{-\frac{22}{15}}(3t^2+16t-4)\geq 3t^{-\frac{22}{15}}(t^2+4t)=3(t^{\frac{8}{15}}+4t^{-\frac{7}{15}})$, and hence we can take $t_0=5^{\frac{15}{8}}$.

Now we give an a priori estimate of the solution away from $x=1$.

\begin{lem}\label{lem_Spinuniformests301}
Let $0<\frac{-8+2\sqrt{19}}{3}+\sigma<t(0)<\frac{11}{3}$ with some constant $\sigma>0$ small. There exists a uniform constant $C=C(\sigma)>0$ independent of the solution and the initial data $t(0)$ such that
\begin{align}\label{ineqn_Spiny1234leftbds}
|y_i^{(k)}(x)| \leq Cx^{2-k},
\end{align}
with $y_i^{(k)}$ the $k-$th order derivative of $x$, for $k=1,2$, $i=1,2$ and $x\in[0,\frac{3}{4}]$. The control still holds on the interval $[0, 1-\epsilon]$ for any $\epsilon>0$ small with some constant $C=C(\sigma,\epsilon)>0$. 
\end{lem}
\begin{proof}
By $(\ref{equn_Spiny1lowerorder21})$ and $(\ref{equn_SpinEinstein02})$, one can easily obtain $(\ref{ineqn_Spiny1234leftbds})$ for $y_1$.
Notice that
\begin{align*}
\min\{1, t(0)\}\leq t(x) \leq t_0,
\end{align*}
for $x\in (0,1)$, and $K(0)< K(x) <1$ for $x\in (0,1)$. By the interior estimates of the second order elliptic equations $(\ref{equn_SpinEinstein01})-(\ref{equn_SpinEinstein03})$ and the inequality $(\ref{ineqn_Spiny2lowerbd1-2})$, there exists some constant $C=C(\sigma)>0$ independent of the initial data and the solution so that
\begin{align*}
|y_i^{(k)}(x)| \leq C,
\end{align*}
for $1\leq i \leq 2$, $0 \leq k\leq 4$ and $x \in [\frac{1}{4}, \frac{3}{4}]$. To get global estimates, we multiply $x^{-14}(1-x^2)^{15}K^{\frac{1}{2}}$ on both sides of the equation $(\ref{equn_SpinEinstein03})$ and do integration on $[x,\frac{3}{4}]$ to have 
\begin{align}
&x^{-14}(1-x^2)^{15}K^{\frac{1}{2}}y_2'=2^{14}(\frac{3}{4})^{15}K^{\frac{1}{2}}(\frac{1}{2})y_2'(\frac{1}{2})\\
&+ 16 \int_x^{\frac{1}{2}}s^{-14}(1-s^2)^{13}K^{\frac{1}{2}-\frac{1}{15}}t^{-\frac{22}{15}}(t-1)(3t-11)ds,\notag
\end{align}
for $x\in(0,\frac{3}{4})$, and hence using the bound of $t$, we have
\begin{align}
\label{ineqn_Spinyd1left302}|y_2'(x)| \leq Cx,
\end{align}
for $x\in(0,\frac{3}{4})$, with some constant $C=C(\sigma)>0$ independent of the solution and the initial data. Substituting $(\ref{ineqn_Spinyd1left302})$ to $(\ref{equn_SpinEinstein03})$, we then have
\begin{align}
&|y_2''(x)| \leq C
\end{align}
for $x\in (0,\frac{3}{4})$, with $C=C(\sigma)>0$ some constant independent of the initial data and the solutions. This completes the proof of the lemma.
\end{proof}

Recall that by the non-positivity of the sectional curvature of $g$ and the Einstein equation $(\ref{eqn_Einstein})$, we have that $|W|_g\leq T\equiv  \sqrt{n(n^2-1)}=4\sqrt{210}$. Using the boundedness of the Weyl tensor, we give an estimate of the solution away from $x=0$.
\begin{lem}\label{lem_SpinyiboundrightGB}
Assume that $|W|_g\leq \varepsilon$ with some constant $0<\varepsilon\leq T$. Let $\delta_0$ be any constant in $(0,1)$.  Under the condition of Lemma \ref{lem_Spinuniformests301}, we have
\begin{align}
&\label{ineqn_Spinyd1right301}|y_1^{(k)}|\leq C\,\varepsilon^2(1-x^2)^{4-k},\\
&\label{ineqn_Spinyd1right302}|y_2^{(k)}|\leq C\,\varepsilon(1-x^2)^{2-k},
\end{align}
for $k=1,2$ and $x \in [\delta_0, 1]$, with some constant $C=C(\sigma,\delta_0)>0$.
\end{lem}
\begin{proof}
By direct calculation, the Weyl tensor satisfies
\begin{align*}
W_{ijk0}=&\frac{1}{4}\frac{d}{dr}g_{jj}[C_{ki}^j-C_{ik}^j+2C_{kj}^ig_{ii}g^{jj}-C_{ji}^k
g_{kk}g^{jj}+C_{ij}^kg_{kk}g^{jj}]+\frac{1}{2}\frac{d}{dr}g_{kk}(C_{ji}^k-C_{ij}^k)\\
&+\frac{1}{2}\frac{d}{dr}g_{ii}[-C_{kj}^i+\frac{1}{2}g^{ii}(-C_{ki}^jg_{jj}+C_{ik}^jg_{jj}-C_{ji}^kg_{kk}+C_{ij}^kg_{kk})].
\end{align*}

Take $i=1,j=10$ and $k=11$. By the condition of the lemma, we have
\begin{align}\label{ineqn_SpinWeylboundmixed}
|(g_{ii})^{-\frac{1}{2}}(g_{qq})^{-\frac{1}{2}}(g_{pp})^{-\frac{1}{2}}W_{piq0}(g)|\leq \varepsilon,
\end{align}
which is
\begin{align*}
&|(g_{ii})^{-\frac{1}{2}}(g_{qq})^{-\frac{1}{2}}(g_{pp})^{-\frac{1}{2}}W_{piq0}(g)|\\
&=|(g_{ii})^{-\frac{1}{2}}(g_{qq})^{-\frac{1}{2}}(g_{pp})^{-\frac{1}{2}}\times\frac{1}{2}g_{11}\frac{d}{dr}(\log(g_{10,10})-\log(g_{11}))|\\
&=\frac{1}{2}|(g_{11})^{-\frac{1}{2}}(g_{10,10})^{-1}\times g_{11}\frac{d}{dr}y_2|\\
&=\frac{1}{2}|(g_{11})^{-\frac{1}{2}}(g_{10,10})^{-1}\times g_{11}(-x\frac{d}{dx}y_2)|\\
&=\frac{x^2}{1-x^2}t^{-1}I_1^{-\frac{1}{2}}|y_2'|\leq \varepsilon.
\end{align*}
Therefore,
\begin{align*}
|y_2'|\leq x^{-2}(1-x^2)K^{\frac{1}{30}}t^{\frac{11}{5}}\varepsilon,
\end{align*}
for $x\in(0,1)$. Therefore, using the bound of $t$,
\begin{align}\label{ineqn_Spiny234-1}
|y_2'|\leq x^{-2}(1-x^2)\delta\varepsilon
\end{align}
for $x\in(0,1)$, with the constant $\delta=t_0^{\frac{11}{5}}$. Set $C=\frac{1}{2\delta_0^2}\delta^{-\frac{3}{5}}$ and hence $(\ref{ineqn_Spinyd1right302})$ holds on $x\in[\delta_0,1]$ for $k=1$. Substituting the inequalities $(\ref{ineqn_Spinyd1right302})$, $(\ref{equn_Spiny1lowerorder21})$, the bound of $K$ and the estimate of $t$ into $(\ref{equn_SpinEinstein03})$, we obtain immediately that there exists a constant $C=C(\sigma,\delta_0)>0$ such that $(\ref{ineqn_Spinyd1right302})$ holds for $k=2$.

For $y_1'$, we multiply $x^{-1}(1-x^2)^2K^{\frac{1}{30}}$ on both sides of $(\ref{equn_SpinEinstein01})$ and do integration on $(x, 1)$ for $\delta_0\leq x \leq 1$ so that
\begin{align*}
&(x^{-1}(1-x^2)^2K^{\frac{1}{30}}y_1')'+\frac{1}{30}x^{-1}(1-x^2)^2K^{\frac{1}{30}}\times 56(y_2')^2=0,\,\,\,\,\text{and}\,\,\\
&y_1'(x)=\frac{28}{15}K(x)^{-\frac{1}{30}}x(1-x^2)^{-2}\int_x^1s^{-1}(1-s^2)^2K^{\frac{1}{30}}(s)\,(y_2')^2ds\\
&\leq C \,\varepsilon^2 (1-x^2)^3,
\end{align*}
with some constant $C=C(\sigma,\delta_0)>0$. Here we have used the estimates of $t$, the monotonicity of $K$ and the inequality $(\ref{ineqn_Spiny234-1})$. Substituting it back to $(\ref{equn_SpinEinstein01})$, we get the estimate $(\ref{ineqn_Spinyd1right301})$ for $y_1''$ with some constant $C=C(\sigma,\delta_0)>0$. This proves the lemma.
\end{proof}

\begin{lem}\label{lem_Spinmonotonicityt}
There exists two constants $\varepsilon>0$ and $\beta>0$ such that if $\sup_M|W|_g\leq \varepsilon$, $t(0)\neq1$ and $1-t(0)\leq \beta$, then $y_2'$ has no zero on $x\in(0,1)$. That is to say, $t$ is monotone on $x\in(0,1)$.
\end{lem}
\begin{proof}
Assume that
\begin{align*}
&\big|1-t(0)\big|\leq \beta,\\
&\sup_M|W|_g\leq \varepsilon,
\end{align*}
with $\beta<\frac{11}{3}-1$ when $t(0)>1$ and $0<\beta<1$ when $t(0)<1$ and $\varepsilon>0$ to be determined. 
By $(\ref{ineqn_Spiny1234leftbds})$, there exists $C_1>0$ independent of $\beta$ and $\varepsilon$ such that
\begin{align*}
|t'|\leq C_1x,
\end{align*}
for $x\in[0,\frac{1}{2}]$, 
and hence,
\begin{align*}
|t(x)-t(0)|\leq \frac{1}{2}C_1x^2
\end{align*}
for $x\in(0,\frac{1}{2})$. Therefore,
\begin{align*}
\min\{t(0),1\}\leq t(x)\leq \frac{11}{3}
\end{align*}
 for
 \begin{align*}
 x\leq x_0\equiv\big(\frac{2(\frac{11}{3}-t(0))}{C_1}\big)^{\frac{1}{2}}.
 \end{align*}
Using the bound $\sup_M|W|_g\leq \varepsilon$ and $(\ref{ineqn_Spinyd1right302})$, we have that
\begin{align*}
|y_2(x_0)|\leq C_2 \varepsilon \int_{x_0}^1(1-s^2)ds= C_2(1-x_0-\frac{1}{3}+\frac{1}{3}x_0^3)\varepsilon.
\end{align*}
for some constant $C_2=C_2(x_0)>0$. Hence, if
\begin{align*}
\varepsilon< \frac{1}{C_2(1-x_0-\frac{1}{3}+\frac{1}{3}x_0^3)}\ln(\frac{11}{3}),
\end{align*}
using the estimate $(\ref{ineqn_Spiny2lowerbd})$ of the local maximum of $t$, we have that there is no local maximum point of $t$ on $x\in(0,1)$. Therefore, by $(\ref{ineqn_Spiny2monotonicity1-1})$ and the boundary condition $(\ref{equn_SpinBV01})$, $y_2'$ has no zero on $x\in(0,1)$. This completes the proof of the lemma.

\end{proof}

We assume that the boundary value problem $(\ref{equn_SpinEinstein01})-(\ref{equn_SpinBV01})$ admits two solutions $(y_{11}, y_{12})$ and $(y_{21}, y_{22})$ with $y_{11}=\log(K_1),\,y_{12}=\log(t_1),\,y_{21}=\log(K_2)\,$ and $y_{22}=\log(t_2)$, for $t(0)\neq 1$ close to $1$. By the same argument in Lemma 5.3 in \cite{Li}, we have 
\begin{lem}\label{lem_zeroSpinz1z2}
Under the condition in Lemma \ref{lem_Spinmonotonicityt}, for any two zeroes $0<x_1<x_2\leq 1$ of $z_1'$ so that there is no zero of $z_1'$ on the interval $x\in(x_1,x_2)$, there exists a point $x_3\in(x_1,x_2)$ so that
\begin{align}\label{inequn_signSpinz201}
(y_{12}'+y_{22}')z_1'z_2'\big|_{x=x_3}<0.
\end{align}
Also, for any zero $0<x_2\leq 1$ of $z_1'$, there exists $\varepsilon>0$ so that for any $x_2-\varepsilon < x <x_2$, we have
\begin{align}\label{inequn_signSpinz202}
(y_{12}'(x)+y_{22}'(x))z_1'(x)z_2'(x)>0.
\end{align}
\end{lem}

It is clear that the function
\begin{align*}
t^{-\frac{22}{15}}(3t^2-14t+11)
\end{align*}
of $t$ decreasing on the interval $t\in(0,t^0)$ with
\begin{align}\label{defn_constant1}
t^0\equiv \frac{\sqrt{49^2+4\times 12 \times 121}-49}{24}=\frac{\sqrt{8209}-49}{24}\sim \frac{41.60353}{24}.
\end{align}
 By the same proof of Theorem 5.4 in \cite{Li}, with the integrating factor $x^{-2}(1-x^2)^3$ in $(5.13)$ in \cite{Li} replaced by $x^{-14}(1-x^2)^{15}$, we have the following uniqueness lemma for the boundary value problem $(\ref{equn_SpinEinstein01})-(\ref{equn_SpinBV01})$. (Notice that Theorem 5.2 in \cite{Li} is not necessary for the uniqueness argument. And for $K_1(0)=K_2(0)$, by the mean value theorem, there exists a zero of $z_1'$ in $x\in(0,1)$, and Theorem 5.4 in \cite{Li} covers this case.)
\begin{lem}\label{lem_SpinmetricStability}
The solution to the boundary value problem $(\ref{equn_SpinEinstein01})-(\ref{equn_SpinBV01})$, with $y_1$ and $y_2$ monotone on $x\in(0,1)$ and $t(0)<t^0$, must be unique if it exists.
\end{lem}

Now we prove the global uniqueness of the conformally compact Einstein metric with $\text{Spin}(9)$ invariant conformal infinity.

\begin{thm}\label{thm_someSpinmetric}
Let $\hat{g}$ be a homogeneous metric on $\mathbb{S}^{15}\cong \text{Spin}(9)/\text{Spin}(7)$ so that $\hat{g}$ has the standard diagonal form
\begin{align}\label{eqn_Spinstandardmetricform}
\hat{g}=\lambda_1(\sigma_1^2+..+\sigma_7^2)+\lambda_2(\sigma_8^2+..+\sigma_{15}^2),
\end{align}
at a point $p\in \mathbb{S}^{15}$, where $\lambda_1$ and $\lambda_2$ are two positive constants and $\sigma_1,..,\sigma_{15}$ are the $1$-forms with respect to the basis vectors in $\mathfrak{p}$, in the $\text{Ad}_{\text{Spin}(9)}$-invariant splitting $spin(9)=spin(7)\oplus \mathfrak{p}$. Assume that $\frac{\lambda_1}{\lambda_2}$ is close enough to $1$, then up to isometry the conformally compact Einstein metric filled in is unique and it is the perturbation metric in \cite{GL} on the $16$-ball $B_1(0)$ with $(\mathbb{S}^{15}, [\hat{g}])$ as its conformal infinity. 
\end{thm}
\begin{proof}[Proof of Theorem \ref{thm_someSpinmetric}]
For the case $t(0)=1$ i.e., $\lambda_1=\lambda_2$ so that the conformal infinity is the round sphere, the theorem has been proved in \cite{Andersson-Dahl}\cite{Q}\cite{DJ}\cite{LiQingShi}.

Now we assume that $\lambda_1\neq \lambda_2.\,$ It is proved in \cite{LiQingShi} and \cite{Cai-Galloway} that for $\frac{\lambda_1}{\lambda_2}$ close enough to $1$, the conformally compact Einstein manifold $(M, g)$ filled in is non-positively curved and simply connected and the closure $\overline{M}=M\bigcup \partial M$ of $M$ is diffeomorphic to the unit ball $\overline{B}_1\subset \mathbb{R}^{16}$. Moreover, the sectional curvature of $g$ is close to $-1$ and hence $\sup_M|W|_g$ is small. Thus, the condition in Lemma \ref{lem_Spinmonotonicityt} is satisfied. Also, let $t(0)<t^0$, with $t^0$ defined in $(\ref{defn_constant1})$.

Pick up a point $q\in \partial M=\mathbb{S}^{15}$. Let $x$ be the geodesic defining function about $C\hat{g}$ with $C>0$ some constant so that $x=e^{-r}$ with $r$ the distance function on $(M,g)$ to the center of gravity $p_0\in M$, see Theorem 3.6 in \cite{Li}.  Under the polar coordinate $(x, \theta)$ with $0\leq x\leq 1$ and $\theta=0$ along the geodesic $\gamma$ connecting $q$ and $p_0$, by \cite{Li2} and Section \ref{Sect:preliminary} we have that  the Einstein equations with prescribed conformal infinity $(\mathbb{S}^{15}, [\hat{g}])$ with $\hat{g}$ the homogeneous metric in $(\ref{eqn_Spinstandardmetricform})$, is equivalent to the boundary value problem $(\ref{equn_SpinEinstein01})-(\ref{equn_SpinBV01})$ along the geodesic $\gamma$ provided that the solution has non-positive sectional curvature. Moreover, since $t(0)$ is close to $1$ and the condition in Lemma \ref{lem_Spinmonotonicityt} is satisfied. By Lemma \ref{lem_Spinmonotonicityt}, $y_1$ and $y_2$ are monotone on $x\in(0,1)$. Moreover, Lemma \ref{lem_zeroSpinz1z2} holds. Then by Lemma \ref{lem_SpinmetricStability}, up to isometries, the CCE metric is unique.
\end{proof}

Using the a priori estimate in Lemma \ref{lem_Spinuniformests301} and Lemma \ref{lem_SpinyiboundrightGB} with $\varepsilon=T$ and Graham-Lee and Lee's perturbation result in \cite{GL} and \cite{Lee}, we obtain directly the following existence theorem by the continuity method as in \cite{Li2}, the proof of which is omitted.

\begin{thm}\label{thm_someSpinmetric1}
Let $B_1\subseteq \mathbb{R}^{16}$ be the unit ball on the Euclidean space with the unit sphere $\mathbb{S}^{15}$ as its boundary. Let $\hat{g}^{\lambda}$ be a homogeneous metric on the boundary $\mathbb{S}^{15}\cong \text{Spin}(9)/\text{Spin}(7)$
so that $\hat{g}$ has the standard diagonal form
\begin{align*}
\hat{g}^{\lambda}=(\sigma_1^2+..+\sigma_7^2)+\lambda(\sigma_8^2+..+\sigma_{15}^2),
\end{align*}
 at a point where $\lambda$ is a positive constant and $\sigma_1,..,\sigma_{15}$ are the $1$-forms with respect to the basis vectors in $\mathfrak{p}$ in the $\text{Ad}_{\text{Spin}(9)}$-invariant splitting $su(9)=su(7)\oplus \mathfrak{p}$. Then as the parameter $\lambda$ varies from $\lambda=1$ continuously on the interval $(\frac{-8+2\sqrt{19}}{3},1]$ (resp. on the interval $[1, \frac{11}{3})$), either it holds that there exists a conformally compact Einstein metric on $B_1$ which is non-positively curved with $(\mathbb{S}^{15}, [\hat{g}^{\lambda}])$ as its conformal infinity for each $\lambda\in (\frac{-8+2\sqrt{19}}{3}, 1]$ (resp. $\lambda\in [1, \frac{11}{3})$); or there exists $\lambda_1\in(\frac{-8+2\sqrt{19}}{3}, 1]$ (resp. $\lambda_1\in[1, \frac{11}{3})$), such that for each $\lambda\in [\lambda_1, 1]$ (resp. $\lambda\in[1,\lambda_1]$) there exists a conformally compact Einstein metric $g^{\lambda}$ on $B_1$ which is non-positively curved with $(\mathbb{S}^{15}, [\hat{g}^{\lambda}])$ as its conformal infinity and there exists $p\in B_1$ such that the sectional curvature of $g^{\lambda_1}$ is zero in some direction at $p$ and moreover, for any $\epsilon>0$ small there exists $\lambda_2\in(\lambda_1-\epsilon, \lambda_1)$ (resp. $\lambda_2\in(\lambda_1, \lambda_1+\epsilon)$) such that there exists a conformally compact Einstein metric $g^{\lambda_2}$ on $B_1$ with $(\mathbb{S}^{15}, [\hat{g}^{\lambda_2}])$ as its conformal infinity and the sectional curvature of $g^{\lambda_2}$ is positive in some direction at some point $p\in B_1$.
\end{thm}
\end{appendix}

 \end{document}